\documentclass[amssymb,10pt]{article}
\usepackage[pagewise]{lineno}
\usepackage{graphicx}
\usepackage{amssymb}
\usepackage{amsthm}
\usepackage{indentfirst}
\usepackage{amsmath}
\usepackage{color}

\usepackage{mathrsfs}
\usepackage{amssymb}
\usepackage{amsthm}
\usepackage{titletoc}
\usepackage{amssymb}
\usepackage{amsxtra}
\usepackage{mathrsfs}
\usepackage{amsmath,amstext,amsfonts,verbatim, url}
\usepackage{epstopdf}
\usepackage{subfig}
\usepackage{psfrag}

\usepackage{tikz}

\usepackage{xy}

\usepackage[colorlinks=true]{hyperref}
\usepackage{amsfonts}
\usepackage{mathrsfs}
\usepackage{amsthm,amsmath,amssymb,mathrsfs,amsfonts,graphicx,graphics,latexsym,exscale,cmmib57,dsfont,bbm,amscd,ulem}
\usepackage{color,soul}
\newtheorem{thm}{Theorem}[section]
\newtheorem{lem}[thm]{Lemma}
\newtheorem{pro}[thm]{Proposition}

\newtheorem{cor}[thm]{Corollary}
\newtheorem{rk}[thm]{Remark}

\newtheorem{claim}[thm]{Claim}
\newtheorem{qn}[thm]{Question}

\def\cal#1{\fam2#1}

\def\R{{\mathbb R}}

\def\C{{\mathbb C}}

\def\E{{\mathbb E}}
\def\T{{\mathbb T}}

\def\bb{\begin}
\def\be{\begin{equation}}
\def\ee{\end{equation}}
\def\bea{\begin{eqnarray}}
\def\eea{\end{eqnarray}}
\def\beaa{\begin{eqnarray*}}
\def\eeaa{\end{eqnarray*}}

\def\ifl{\iffalse}

\def\bb{\begin}
           \def\ea{\end{array}}
          \def\ec{\end{center}}
     \def\ed{\end{description}}
\def\be{\bb{equation}}        \def\ee{\end{equation}}
\def\bea{\bb{eqnarray}}       \def\eea{\end{eqnarray}}
\def\beaa{\bb{eqnarray*}}     \def\eeaa{\end{eqnarray*}}
 \def\et{\end{thebibliography}}

   \def\E{{\cal E}}    \def\F{{\cal F}}    
   \def\U{{\cal U}}        \def\P{{\cal P}}
          
\def\C{\mathcal{C}}

       \def\nt{\noindent}

\def\Int{{\rm Int}}

\def\loc{{\rm loc}}

\def\Orb{{\rm Orb}}

\def\supp{{\rm supp}}

\def\Leb{{\rm Leb}}
\def\top{{\rm top}}
\def\vol{{\rm vol}}

\begin{document}

\title{Finiteness and Geometric structure of $c$-$cu$-States with Maximal $u$-Entropy}

\date{}
\maketitle

{\center
Hangyue Zhang

\smallskip

Department of Mathematics

Nanjing University

Nanjing 210093, China

\smallskip

\smallskip


\smallskip

}

\begin{abstract}
In a $c$-mixed system, we study $c$-$cu$-states, which capture the structural characteristics of physical measures (in similar systems), having maximum $u$-entropy. It is shown that the maximum number of $c$-$cu$-states with pairwise distinct supports is finite, and  Proposition~\ref{pro.con} is provided to construct such systems. Using a modified version of Smale's method \cite{Smale}, we explicitly construct a \( C^{\infty} \) diffeomorphism \( f \) on \( \mathbb{T}^4 \) with a partially hyperbolic splitting:  
\[
F^{uu} \oplus_{\succ} F^{cu} \oplus_{\succ} (F^{cs} \oplus_{\succ} F^{ss}),
\]  
such that \( f \) has a mixed center (or \( c \)-mixed center), \( F^{cu} \) is not uniformly expanding, and \( F^{cs} \oplus F^{ss} \) is not uniformly contracting. The method can be used to modify the product maps of linear Anosov skew products and linear Anosov systems, such that the modified map has a mixed center (or \( c \)-mixed center) and is a skew product of linear Anosov skew product. This provides concrete examples to illustrate how the physical measure changes in a semicontinuous manner across the system when the corresponding \( E^{cu} \) is non-uniformly expanding and the corresponding \( E^{cs} \) is non-uniformly contracting. The study of physical measures in similar systems can be found in the literature \cite{ref7, CM}.
\end{abstract}

\section{Introduction}
Since  Kolmogorov introduced the concept of metric entropy of invariant measures in 1958, metric entropy has become a cornerstone of ergodic theory. In 1965,  Adler, Konheim, and McAndrew \cite{AKM} expanded this field by introducing topological entropy, which provides insight into the exponential expansion of orbital segments. These two forms of entropy are intimately connected through the well-known variational principle. In the 1970s, researchers such as Dinaburg \cite{Din1,Din2}, Goodman \cite{Gom}, and Goodwin \cite{Gow}  demonstrated  that the topological entropy is equal to  the supremum of the metric entropies of invariant measures on a compact metric space. This establishes a quantitative link between the two types of entropy, although it remains possible that no invariant measures achieve a metric entropy equal to the topological entropy.

In 1985,  Ledrappier and Young \cite{LY1,LY2} introduced the concept of partial entropy along an invariant lamination, providing a new statistical perspective on the complexity of expansive components in dynamical systems. 

Let $M$ be a compact smooth connected Riemannian manifold with a Riemannian metric $\|\cdot\|$.   A diffeomorphism $f$ on a manifold $M$ is called a {\it partially hyperbolic diffeomorphism} if there exists a continuous $Df$-invariant splitting of the tangent bundle:    
$$
TM=E^{uu}\oplus E^{cu} \oplus E^{cs}
$$ 
where $E^{uu}$, $E^{cu}$, and $E^{cs}$ are subbundles satisfying the following conditions for some constants $0 < \lambda < 1$ and $c > 0$: 
for every point $x \in M$ and any natural number $n \in \mathbb{N}$, the following conditions hold:
\begin{itemize}
\item $\|Df^{-n}\mid_{E^{uu}_x} \|\le c\lambda^n$, , which implies that $E^{uu}$ is an {\it expanding subbundle};

\item  $\| Df^{n}\mid_{E^{cu}_x}\|\cdot\| Df^{-n}\mid_{E^{uu}_{f^n(x)}}\|\le c\lambda^{n} $, ;

\item  $\| Df^{n}\mid_{E^{cs}_x}\|\cdot\| Df^{-n}\mid_{E^{cu}_{f^n(x)}}\|\le c\lambda^{n}$.
\end{itemize}
This splitting is often referred to as a {\it partially hyperbolic splitting}, denoted by $E^{uu}\oplus_{\succ} E^{cu} \oplus_{\succ} E^{cs}$, where the symbol “$\succ$” indicates that the preceding subbundle dominates or controls the following subbundle.

It is well known that the partially hyperbolic splitting implies the existence of a unique $f$-invariant foliation $\mathcal{F}^{uu}(f)$ tangent to $E^{uu}$ at every point. Denote by $\F_\delta^{uu}(x,f)$ the open ball inside $\F^{uu}(x,f)$ centered at $x$ with  radius $\delta$, using the metric of $\F^{uu}(x,f)$.

In 2008,  Hua, Saghin, and Xia\cite{HSX} defined topological partial entropy along the strong unstable foliation of partially hyperbolic diffeomorphisms. The {\it topological partial entropy}, or {\it topology $u$-entropy}, of $f$  is defined as 
$$
h_\top^u(f)=\sup_{x\in M}\limsup_{n\to \infty} \dfrac{1}{n}\log\vol(f^n(\F_\delta^{uu}(x,f))),
$$
where $h_\top^u(f)$ is independent of the choice of  $\delta$. They discovered that topological partial entropy exhibits properties analogous to those of topological entropy.
Their findings, later reinforced by Hu, Hua, and Wu \cite{HHW}  in 2017, established that this topological partial entropy is the supremum of the partial entropies of invariant measures along the same foliation. 

 Let $\mu$ be an invariant measure of $f$.
The {\it partial entropy}, or {\it $u$-entropy}, of $\mu$  is defined as
$$
h_\mu^u(f)=H_\mu(\xi|f(\xi)),
$$
where $\xi$ is any  measurable partition  of $M$ satisfying the following conditions for $\mu$-almost every $x$:
\begin{itemize}
	\item $\xi(x)$ is contained in $\F^{uu}(x,f)$ and has uniformly small diameter with respect to the metric in $\F^{uu}(x,f)$;
	\item $\xi(x)$ contains an open neighborhood of $x$ in $\F^{uu}(x,f)$;
	\item $\xi$ is a refinement of $f(\xi)$.
\end{itemize}
It follows from \cite{LY1} that  $H_\mu(\xi|f(\xi))$ is independent of the choice of $\xi$  as long as $\xi$ satisfies  three conditions  above.

 Building on the variational principle  of $u$-entropies,  Hu, Wu, and Zhu \cite{HWZ} confirmed  in 2021 the existence of invariant measures whose partial entropies coincide with the topological partial entropy. We say that  $\mu$ is  a  {\it measure of maximal $u$-entropy} for $f$ if the partial entropy of $\mu$  equals to the topology partial entropy of $f$.

On the other hand, the {\it basin} of $\mu$ is defined as:
$$
B(\mu)=\{x\in M: \mu=\lim_{n\to+\infty} \frac{\sum_{0\le j\le n-1}\delta_{f^j(x)}}{n} \mbox{ in weak$^*$ topology}\}.
$$
If the Lebesgue measure of $B(\mu)$ is positive, then $\mu$ is called a {\it  physical measure}. The concept of physical measures was introduced by Sinai, Ruelle, and Bowen \cite{Bow,BR,am,gi} while studying Anosov systems and Axiom A attractors in the 1970s. Since dynamical systems primarily focus on the limit distribution of  orbits, physical measures  are crucial for analyzing the orbital behavior of  points with full-volume. Consequently, the existence, finiteness, and variation of physical measures in systems exhibiting various forms of non-uniform hyperbolicity have become popular topics, as seen in \cite{BV,DVY,ja1,ja3,ja2,cv,yca,and,CM,hua}.

This study investigates  physical measures and maximal $u$-entropy measures for partially hyperbolic diffeomorphisms.

In cases where \( E^{cs} \) is uniformly contracting, \( f \) is called {\it non-uniformly expanding along \( E^{cu} \)} if there exists a positive Lebesgue measure set \( H \) such that for every \( x \in H \),  
\[
\limsup_{n \to +\infty} \frac{1}{n}\sum_{j=1}^n\log\|Df^{-1}|_{E^{cu}_{f^j(x)}}\| < 0.
\]  
This class of partially hyperbolic systems, which includes the case where \( E^{uu} \) is trivial, was investigated by Alves, Bonatti, and Viana in   \cite{ABV} in 2000. They established the existence of physical measures for such systems.

In 2000, Bonatti and Viana \cite{BV} discovered that in cases where \( E^{cu} \) is a trivial bundle, partially hyperbolic diffeomorphisms can admit finitely many physical measures when \( E^{cs} \) is mostly contracting. Later, in 2016, Dolgopyat, Viana, and Yang \cite{DVY} further advanced the understanding of these systems by characterizing the supports and basins of the physical measures. They introduced the concept of a skeleton, which revealed the upper semi-continuity of the number of physical measures in such systems.

The study of physical measures in partially hyperbolic systems originates from the works \cite{ABV,BV} mentioned in the above two paragraphs.

In 2020, Ures, Viana, F. Yang, and J. Yang \cite{UVYY} employed the concept of a skeleton to show that for partially hyperbolic diffeomorphisms that factor over Anosov with the splitting \( TM = E^{uu} \oplus_{\succ} E^{cs} \), there exist finitely many ergodic measures of maximal \( u \)-entropy when \( E^{cs} \) is \( c \)-mostly contracting. Building on these results, Li and the author \cite{LiZhang} proved in 2023 that the number of ergodic measures of maximal \( u \)-entropy is upper semi-continuous among such diffeomorphisms factoring over the same Anosov system.

For a partially hyperbolic splitting $TM=E^{uu}\oplus_\succ E^{cu}\oplus_\succ E^{cs}$,  the subbundle $E^{cu}$ is said to be  {\it mostly expanding  in the strong sense} if, for every disk \( D^{uu} \) contained in a leaf of the strong unstable foliation, there exists a subset \( D^{uu}_0 \subset D^{uu} \) of positive Lebesgue measure such that for any \( x \in D^{uu}_0 \),  
$$
\limsup_{n\to+\infty}\frac{1}{n}\log\|Df^{-n}|_{E^{cu}_{f^n(x)}}\|<0.
$$
This definition was  given by Andersson and V\'{a}squez\cite{AV}, who demonstrated that if $E^{cu}$ is mostly expanding  in the strong sense and $E^{cs}$ is uniformly contracting, then any Gibbs $u$-state has only positive Lyanunov exponents along $E^{cu}$ and there are finite physical measures in 2018. They explained that Ma\~{n}\'{e}’s classical derived-from-Anosov diffeomorphism on $\T^3$ belongs to this set.

In 2017, Mi, Cao, and Yang \cite{ref7} generalized the assumptions on the partially hyperbolic splitting studied in  \cite{BV, AV}.  	They established that partially hyperbolic diffeomorphisms with the splitting \( TM = E^{uu} \oplus_{\succ} E^{cu} \oplus_{\succ} E^{cs} \), where \( E^{cu} \) is mostly expanding and \( E^{cs} \) is mostly contracting, admit finitely many physical measures. Later, in 2020, Mi and Cao \cite{CM} further characterized the supports of these physical measures by utilizing the concept of skeletons, revealing the upper semi-continuity of the number of physical measures among such diffeomorphisms.  
We refer to \( E^{cu} \oplus_\succ E^{cs} \) as the {\it center subbundle}.
For simplicity,  \( E^{cu} \) is mostly expanding and \( E^{cs} \) is mostly contracting,  is referred to as {\it a mixed center}.  (The precise definitions of mixed centers and mostly contracting centers can be found in section~\ref{nonapp}.)

We explore the possibility of employing a skeleton to analyze the number and variation of ergodic measures of  maximal $u$-entropy for partially hyperbolic diffeomorphisms factoring over Anosov. These diffeomorphisms are characterized by the splitting $TM=E^{uu}\oplus_{\succ} E^{cu} \oplus_{\succ} E^{cs}$, where $E^{cu}$ is $c$-mostly expanding and $E^{cs}$ is $c$-mostly contracting.  In comparison to earlier work, although the assumptions we impose on the diffeomorphisms are more restrictive—such as requiring the system to factor over Anosov—this specific condition allows us to identify concrete examples within this framework. As mentioned below, the measure of maximal \( u \)-entropy and Gibbs \( u \)-states are closely related through the factorization over Anosov systems. Many methods are also applicable for constructing examples with a mixed center.  For clarity,  a center that is both $c$-mostly expanding and $c$-mostly contracting, as  mentioned  in this paragraph, is referred to as {\it a $c$-mixed center}.

 The motivation for this research stems from the ongoing studies of  measures of maximal $u$-entropy and physical measures in the context of partially hyperbolic diffeomorphisms under various central conditions.
Through a comprehensive analysis, we recognize the necessity of establishing a proper classification for these ergodic measures of maximal $u$-entropy. 

It is noteworthy that in $C^{1+}$-partially hyperbolic systems with  mixed centers, physical measures directly correspond  to ergodic Gibbs $cu$-states. Consequently, the number of physical measures is equal to the number of ergodic Gibbs $cu$-states,  as shown in previous work. Some subset of Gibbs $u$-states is the set of ergodic Gibbs $cu$-states. When investigating physical measures, it is frequently necessary to apply the properties of Gibbs $u$-state, which are analogous to those outlined in Proposition~\ref{main.pro}. 

Moreover, we have observed that for partially hyperbolic diffeomorphisms factoring over Anosov, the  measures of maximal $u$-entropy are equivalent to the  $c$-Gibbs $u$-states, which possess properties similar to those of Gibbs $u$-states (see Proposition~\ref{main.pro} for further details). Since the number of ergodic Gibbs $cu$-states matches the number of physical measures in systems with a mixed center, and because physical measures can analyze the orbit distribution of full-volume points with respect to the  Lebesgue measure, it is natural  to prioritize maximal $u$-entropy measures whose structure is  similar to that of physical measures.

Local Pesin unstable manifolds are a powerful tool for investigating the existence and finiteness of physical measures in the $C^{1+}$-setting, particularly under conditions of non-uniform expansion along an invariant center-unstable subbundle, as demonstrated in \cite{ja2,ja3,AV,cv,ABV}.

In this context,  local Pesin unstable manifolds are no longer  merely  a tool for studying physical measures in the sense of non-uniform expansion. Instead, we utilize them to classify ergodic measures of maximal $u$-entropy into two distinct types: $c$-$u$-state and $c$-$cu$-state, among partially hyperbolic diffeomorphisms factoring over Anosov with $c$-mixed centers. 
We have abstracted and extracted a more general structural property of Gibbs $cu$-states to define the $c$-$cu$-state (as detailed in  the definition of $c$-$cu$-state). The notion of $c$-$cu$-states is essentially different from that of Gibbs $cu$-states. 
To clarify this distinction further, we point out that the conditional measures associated with $c$-$cu$-states along local Pesin unstable manifolds are unclear.  This critical distinction emphasizes the absence of conditional measure characteristics inherent to $c$-$cu$-states when compared to Gibbs $cu$-states, which are known to have conditional measures that are equivalent to the Lebesgue measure on local Pesin unstable manifolds almost everywhere.  Furthermore, stable lamination is absolutely continuous when the partially hyperbolic diffeomorphism is $C^{1+}$.

After defining $c$-$cu$-states, we investigate the existence and finiteness of these measures. Although there are no general theoretical results confirming their existence,  we can demonstrate their existence through qualitative constructions.

Our proposed  definition of a skeleton differs from those  presented in previous works \cite{DVY} and \cite{CM},  yet it is rooted in the observation that  the support structures of physical measures(ergodic Gibbs $cu$-states) and ergodic measures of maximal $u$-entropy of type $c$-$cu$-states are closely related to invariant manifolds of hyperbolic periodic points. Our  results suggest that the maximum number of 
$c$-$cu$-states with pairwise distinct supports is finite
 among partially hyperbolic diffeomorphisms factoring over Anosov with  $c$-mixed centers. We point out that the finiteness established by Theorem A is obatined under a very weak definition of $c$-$cu$-states and differs from the methods used to establish the finiteness of physical measures \cite{ref7}.

 Furthermore, we have characterized the support and basin of  $c$-$cu$-states using our defined skeleton (see Theorem~A for more details). The property of having a $c$-mixed center is open in the sense of factoring over the same Anosov (see Theorem~ B). Additionally,  the maximum number of $c$-$cu$-states whose supports are pairwise distinct is locally bounded among partially hyperbolic diffeomorphisms factoring over the same Anosov with  $c$-mixed centers.

In the study of partially hyperbolic systems, researchers are often concerned with the availability of sufficiently many examples related to the objects of study. Proposition~\ref{pro.con} plays a crucial role in this regard, serving two  purposes in our work.
 First, it follows from Proposition~\ref{pro.con} that a partially hyperbolic diffeomorphism with the splitting \( TM = E^{uu} \oplus_{\succ} E^{cs} \) and a \( c \)-contracting center (or contracting center) can always be used to construct a partially hyperbolic diffeomorphism with a \( c \)-mixed center (or mixed center), where the corresponding \( E^{cu} \) is uniformly expanding. Examples of diffeomorphisms with mostly contracting centers have been discovered by several authors. Examples include: 
\begin{itemize}
\item Ma\~{n}\'{e}'s  robustly transitive diffeomorphisms (see  \cite{Man,BV,BDV});
\item Dolgopyat's \cite{Dolgopyat} volume-preserving perturbations of time-one maps of Anosov flows;
\item Volume-preserving diffeomorphisms exhibiting negative center Lyapunov exponents and minimal unstable foliations, as discussed in \cite{VY,BV,BDP,kbu};  
\item Accessible skew-products over Anosov on  $M \times S^1$  that are not rotation extensions, described in \cite{VY}.  
\item New Kan-type skew-products on $\T^2\times S^2$ of Dolgopyat, Viana and Yang \cite{DVY}.
\end{itemize}  
All these examples can be adapted to construct systems with mixed centers by Proposition~\ref{pro.con}. Secondly,  it  can be used  to construct an example showing the upper semi-continuous variation in the number of physical measures among partially hyperbolic diffeomorphisms with the nontrivial splittings
 $TM=E^{uu}\oplus_{\succ} E^{cu} \oplus_{\succ} E^{cs}$ and   mixed centers.
We use  examples of New Kan-type skew-products to show the upper semi-continuity of physical measures in mixed systems. However, in this case, the corresponding \( E^{cu} \) is uniformly expanding. (See the section~\ref{qual} for more information).

In the final section, we construct a new type of example: a \( C^\infty \)-partially hyperbolic diffeomorphism \( f \) on \( \mathbb{T}^4 \) that admits a partially hyperbolic splitting 
\[
F^{uu} \oplus_{\succ} F^{cu} \oplus_{\succ} (F^{cs}\oplus_\succ F^{ss}).
\]
The diffeomorphism \( f \) has the property that every Gibbs \( u \)-state (or \( c \)-Gibbs \( u \)-state) exhibits positive Lyapunov exponents along \( F^{cu} \) and negative Lyapunov exponents along \( F^{cs} \oplus F^{ss} \). Additionally, $F^{cu}$ is not uniformly expanding and $F^{cs} \oplus F^{ss}$ is not uniformly contracting. As an application, a partially hyperbolic diffeomorphism  constructed by Proposition~\ref{pro.con} can be further modified in a similar manner such that the modified map's  corresponding \( E^{cu} \) is not uniformly expanding and the corresponding \( E^{cs} \) is not uniformly contracting, along with a domination (\(E^{cu}\oplus_\succ E^{cs}\)).   For instance, in the final section, we present an example illustrating how a linear Anosov skew-product coming from Proposition~\ref{pro.con} can be modified. Based on this, we can also provide examples of the semicontinuous variation of physical measures in mixed systems, where \( E^{cu} \) is not uniformly expanding and \( E^{cs} \) is not uniformly contracting.  The result on the semicontinuity of physical measures in mixed systems is derived from \cite{CM}.

By the way, it appears that, up to now, no examples of partially hyperbolic diffeomorphisms with mixed centers have been found, where the corresponding \( E^{cu} \) is not uniformly expanding and the corresponding \( E^{cs} \) is not uniformly contracting. (In each section, the partially hyperbolic diffeomorphism represented by \( f \) varies, but we will explicitly specify its meaning in each case.)

\section{Definitions and Main Results}\label{2}

Throughout this section, we assume that $f:M\rightarrow M$ is a  partially hyperbolic diffeomorphism with a  splitting $TM=E^{uu}\oplus_{\succ} E^{cu} \oplus_{\succ} E^{cs}$, and $\mu$ is an invariant measure of $f$.

\subsection{Measures of Maximal $u$-Entropy for Maps Factoring over Anosov}

Let \( A: \mathbb{T}^d \rightarrow \mathbb{T}^d \) be a hyperbolic linear automorphism, and let \( W^{u}(A) \) and \( W^{s}(A) \) denote the unstable and stable foliations of \( A \), respectively. We say that $\mathcal{R}=\{\mathcal{R}_{1},\cdots,\mathcal{R}_{m}\}$ is a {\it Markov partition} for $A$   if the following conditions are satisfied:
\begin{itemize}
\item $\mathcal{R}=\{\mathcal{R}_{1},\cdots,\mathcal{R}_{m}\}$ is
a  closed covering of $\mathbb{T}^{d}$ such that 
\begin{itemize}
\item each set $\mathcal{R}_i$ is equal to the closure of its interior;
\item the interiors of distinct sets in the partition are disjoint.
\end{itemize}
\item For each $\mathcal{R}_i$, and any point $x\in\mathcal{R}_i$, denote by $W^{u/s}_{i}(x)$   the connected component  of  $W^{u/s}(x,A)\cap\mathcal{R}_i$ that contains $x$, respectively. Then:
{\begin{itemize}
\item for any $x,y\in\mathcal{R}_{i}$, $W^{u}_{i}(x)\cap W^{s}_{i}(y) $ contains exactly one point;
\item if $x\in \Int(\mathcal{R}_i)\cap A^{-1}(\Int(\mathcal{R}_{j}))$, then $A(W^{s}_{i}(x))\subset W^{s}_{j}(A(x))$ and $A(W^{u}_{i}(x))\supset W^{u}_{j}(A(x))$.
\end{itemize}}
\end{itemize}
It has been shown by Bowen in \cite{Bow} that Markov partitions always exist for $A$.

 We say that  $f$ is    {\it dynamically coherent}  if  there exists a unique $f$-invariant foliation $\F^c(f)$ tangent to $E^{cu} \oplus E^{cs}$  at every point. Recall that $\mathcal{F}^{uu}(f)$ is the unique $f$-invariant foliation  tangent to $E^{uu}$ at every point.  
Assume that $f$ is dynamically coherent. Now, we say that $f$ {\it factors over Anosov (or $A$)} if  there exists  a continuous and surjective map $\pi:M\rightarrow \mathbb{T}^d$ satisfying the following conditions:
\begin{itemize}
\item Semiconjugacy: $\pi\circ f=A\circ\pi$;
\item Unstable foliation preservation: for any $x\in M$, $\pi(\F^{uu}(x,f))=W^u(\pi(x),A)$ and $\pi$  is a homeomorphism when restricted to $\F^{uu}(x,f)$;
\item  Center foliation under $\pi$: for any $x\in M$, $\pi(\F^c(x,f))=W^s(\pi(x),A)$.
\end{itemize}

Assume that $f$ factors over Anosov via the map $\pi$. We define a
{\it $\pi^{-1}$-Markov partition} of $f$ as
 $$
\mathcal{M}=\{\mathcal{M}_i: \mathcal{M}_i=\pi^{-1}(\mathcal{R}_i), 1\le i\le k\}.
$$
For each $x$,  let $\F^{uu}_i(x)$ denote the connected component of $\F^{uu}(x)\cap\mathcal{M}_i$ that contains $x$.  This $\F^{uu}_i(x)$ is called the {\it strong-unstable plaque} of $x$.  It has been shown by Ures,  Viana, F. Yang and J. Yang in \cite{UVYY}  that $$\pi(\F^{uu}_i(x))=W^u_i(\pi(x)).$$ A  probability measure $\nu^{uu}_{i,x}$ on $\F^{uu}_i(x)$ is called a {\it reference measure}   if  $$\pi_*(\nu^{uu}_{i,x})=\vol^u_{i,\pi(x)},$$
	where $\vol^u_{i,\pi(x)}$ is the normalized Lebesgue measure on $W^u_i(\pi(x))$. 
	
We say that $\mu$ is  a {\it $c$-Gibbs $u$-state} of $f$ if, for each $i$, $\mu$-almost every $x\in \mathcal{M}_i$, the Rokhlin disintegration of the restriction $\mu|_{\mathcal{M}_i}$ along $\F^{uu}_i(x)$  coincides with the reference measure $\nu^{uu}_{i,x}$.  Let \( \mathrm{Gibbs}^u_c(f) \) and \( \mathrm{EGibbs}^u_c(f) \) denote the sets of all \( c \)-Gibbs \( u \)-states and all ergodic \( c \)-Gibbs \( u \)-states of \( f \), respectively. Ures, Viana, F. Yang, and J. Yang proved in \cite{UVYY} that the set of measures of maximal \( u \)-entropy coincides with \( \mathrm{Gibbs}^u_c(f) \).

\subsection{The
$c$-Mixed Center, Classification of Ergodic Measures of Maximal 
$u$-Entropy, and Skeleton}

Assume that \( f \) factors over Anosov. We say that \( f \) has a {\it \( c \)-mostly expanding and \( c \)-mostly contracting center} (or a {\it \( c \)-mixed center}) if every \( c \)-Gibbs \( u \)-state has only positive Lyapunov exponents along \( E^{cu} \) and only negative Lyapunov exponents along \( E^{cs} \). In this subsection, we assume that \( f \) has a \( c \)-mixed center.

We say $\mu$   a {\it $c$-$cu$-state} of $f$ if $\mu$ is an ergodic $c$-Gibbs $u$-state, and there exists a measurable subset $\Gamma(\mu)$ with $\mu(\Gamma(\mu)) > 0$, along with a measurable function $\delta: \Gamma(\mu) \to \mathbb{R}^+$. For every $x \in \Gamma(\mu)$, the condition 
$$
W^u_{\delta(x)}(x) \subset \supp(\mu)
$$ 
holds, where $W^u_{\delta(x)}(x)$ denotes the open ball centered at $x$ with radius $\delta(x)$ within the local Pesin unstable manifold $W^u_\loc(x)$, using the metric of $W^u_\loc(x)$. The dimension of the manifold $W^u_{\text{loc}}(x)$ is given by $dim(E^{uu}\oplus E^{cu})$(where  $W^u_\loc(x)$ as described in Lemma~\ref{kakaka}). Now, $\mu$ is called a {\it $c$-$u$-state} of $f$ if $\mu$ is an ergodic $c$-Gibbs $u$-state but not a $c$-$cu$-state.

Let $EG^{cu}(f)$ denote the set of all  $c$-$cu$-states, and let $EG^{u}(f)$ denote the set of  all  $c$-$u$-states. It follows that $$EGibbs^u_c(f)=EG^{cu}(f)\cup EG^{u}(f).$$ To avoid confusion, 
set $EGibbs^u_c(f)=EG^{cu}(f)$  when $\dim(E^{cu})=0$.

Consider a finite subset $T$ of the manifold $M$, defined by
$$
T=\{p_1,\cdots,p_k| \mbox{ each $p_i$ is a hyperbolic periodic point of stable index $dim E^{cs}$}\}.
$$ 
We call $T$  a  {\it skeleton} of $f$ if for any $i\ne j$, the following conditions do not  hold simultaneously:
\begin{itemize}
\item $W^u(p_j)\cap W^s(p_i)\neq\emptyset$;
\item $W^u(p_i)\cap W^s(p_j)\neq\emptyset$.
\end{itemize}
Here, $W^{u/s}(\cdot)$ denote the unstable  and  stable manifolds through the point $(\cdot)$, respectively.

\subsection{Statements of the Main Results}

\smallskip
\smallskip
\smallskip
\smallskip

\nt {\bf Theorem~A.} {\it Let $f$  factor over Anosov with a $c$-mostly expanding  and $c$-mostly contracting center. If the set of maximal $u$-entropy measures of type $c$-$cu$-states is not empty, then the following statements hold:
\begin{enumerate}
\item The maximum number of $c$-$cu$-states with pairwise distinct supports is finite.
\item There exists a skeleton $T(f)=\{q_1,...,q_\ell\}$ such that
\begin{enumerate}
\item For any $c$-$cu$-state $\mu$, there exists some periodic point $q_j$ in the skeleton such that
\begin{itemize}
\item the support of $\mu$ is the closure of the unstable manifolds through the orbit of  $q_j$;
\item the closure of the basin of $\mu$ contains the closure of the stable manifolds through the orbit of  $q_j$.
\end{itemize}
\item For each $q_i$ in the skeleton, there exists some $c$-$cu$-state $\nu$ such that 
\begin{itemize}
\item the closure of the unstable manifolds through the orbit of $q_i$ is the support of  $\nu$;
\item the closure of the stable manifolds through the orbit of $q_i$ is contained in the closure of the basin of  $\nu$.
\end{itemize}
\end{enumerate}
\end{enumerate}

As a result, the maximum number of $c$-$cu$-states whose supports are pairwise distinct does not exceed $\ell$.}

\smallskip
\smallskip
\smallskip
\smallskip

For clarity, we point out that the result of Theorem A states that among any $\ell + 1$ $c$-$cu$-states, at least two have the same support.  Since the maximum number of $c$-$cu$-states with pairwise distinct supports is finite, we can further consider the changes in this number.
The second part of Theorem~A describes the geometric structure of $c$-$cu$-states, enabling us to track the number of these measures as the given  diffeomorphism is perturbed.

Furthermore, we obtain the following  theorems. 

\smallskip
\smallskip
\smallskip
\smallskip

\nt {\bf Theorem~B.} {\it Let $A:\mathbb{T}^d\rightarrow \mathbb{T}^d$ be a hyperbolic linear automorphism.
If $f$ factors over $A$ with a $c$-mostly expanding  and $c$-mostly contracting center. Then there exist a $C^1$-neighborhood $\mathcal{U}_f$ of $f$ and a constant $\tilde{\ell}$ such that any $g\in \mathcal{U}_f$ factoring over $A$,   the following conditions hold:
\begin{itemize}
\item $g$ admits a $c$-mostly expanding  and $c$-mostly contracting center.
\item  The maximum number of $c$-$cu$-states of $g$ whose supports are pairwise distinct is not bigger than $\tilde{\ell}$.
\end{itemize}
}

\smallskip
\smallskip
\smallskip
\smallskip

\nt {\bf Theorem~C.} {\it  There exist a $C^\infty$-diffeomorphism $f$ on $\mathbb{T}^2\times\mathbb{T}^2\to \mathbb{T}^2\times\mathbb{T}^2$, a \( C^1 \)-neighborhood $\U_f$ of \( f \) and a hyperbolic linear automorphism $A:\mathbb{T}^2\rightarrow \mathbb{T}^2$  such that for each $g\in\U_f$
\begin{itemize}
\item\label{(111)} $g$ factors over $A$.
\item\label{(222)} $g$ admits a partially hyperbolic splitting $T(\mathbb{T}^2\times\mathbb{T}^2)=E^{uu}_g\oplus_\succ E^{cu}_g\oplus_\succ E^{cs}_g$ such that 
\begin{enumerate}
\item every $c$-Gibbs $u$-state of $g$ has only positive Lyapunov exponents along $E^{cu}_g$ and only negative Lyapunov exponents along $E^{cs}_g$;
\item $E^{cu}_g$ is not uniformly expanding and $E^{cs}_g$ is not uniformly contracting.
\end{enumerate}
\end{itemize}
}

\section{Lemmas to Prove Main Theorems}

\smallskip
\smallskip

\subsection{Properties of Hyperbolic Measures and Invariant Measures}

In this subsection, we present the major tools used throughout this article.    Let $f:M\rightarrow M$  be a partially hyperbolic diffeomorphism with the splitting $TM=E^{uu}\oplus_{\succ}E^{cu} \oplus_{\succ}E^{cs}$. Define $E=E^{uu}\oplus E^{cu}$ and $F=E^{cs}$.  The {\it Pesin blocks} are defined as:
\[\Lambda_f(\alpha,l,E,F)=\{x: \prod_{i=0}^{n-1}\|Df^l|_{F(f^{il}(x))}\le e^{-\alpha nl}, \prod_{i=0}^{n-1}\|Df^{-l}|_{E(f^{-il}(x))}\le e^{-\alpha nl},  \forall n\in\mathbb{N}\},\]
where $l\in\mathbb{N}, \alpha>0$.

Any point in a Pesin block $\Lambda_f(\alpha,l,E,F)$ admits stable and unstable manifolds of  uniformly  size.  Mi and Cao, using the Plaque Family Theorem, presented the following lemma.  The property of having a partially hyperbolic splitting is open among   diffeomorphisms. 

\begin{lem}\label{kakaka}\cite{CM} 
For every $\alpha>0$ and $l\in\mathbb{N}$, there exist a $C^1$ neighborhood $\mathcal{U}$ of $f$,  along with constants $\tau:=\tau(\alpha, l)\in(0,1), C:=C(\alpha,l)>0$, and $\delta:=\delta(\alpha,l)>0$. These constants guarantee the following properties: for any diffeomorphism $g\in\mathcal{U}$ and any $x\in \Lambda_g(\alpha,l,E_g,F_g)$, there exist local stable
manifold $W^s_{\loc}(x,g)$    and local unstable manifold $W^u_{\loc}(x,g)$, tangent to the bundles $F_g$ and $E_g$ at every point respectively. 

These manifolds are $C^1$-embedded disks of radius $\delta$ centered at $x$,  using the metric of these manifolds.  The following properties hold for every  $n\in\mathbb{N}$:
\begin{itemize}
	\item $d(g^n(y),g^n(z))\le C\tau^nd(y,z)$ for any $y,z\in W^s_{\loc}(x,g)$;
	\item $d(g^{-n}(y),g^{-n}(z))\le C\tau^nd(y,z)$ for any $y,z\in W^u_{\loc}(x,g)$.
\end{itemize}

Moreover,  there exist constants $C^\prime$ and $\varepsilon_0>0$ such that for any disk $D$ tangent to $E_g$ at every point, the metric $d_D$ inherited from $E_g(x)(\varepsilon_0)$ is equivalent to  the ambient metric $d$ on the manifold $M$. This equivalence is given by:
$$
\frac{1}{C^\prime}d(z,y)\le d_D(z,y)\le C^\prime d(z,y)
$$
for any points $y,z\in D\cap \exp_x(E_g(x)(\varepsilon_0))$, where $\exp_x:E_g(x)\mapsto M$ is the exponential map at $x$. Additionally, both $g(W^u_{\loc}(x, g))$ and $g^{-1}(W^u_{\loc}(x, g))$ remain tangent to $E_g$.
\end{lem}

This lemma states that for  any $a,b\in\Lambda_g(\alpha,l,E_g,F_g)$, if $a$ is sufficiently close  to $b$, then $W^s_{\loc}(a)$ transversely intersects $W^u_{\loc}(b)$, and $W^u_{\loc}(a)$ transversely intersects $W^s_{\loc}(b)$. For simplicity, we  abbreviate $\Lambda_g(\alpha,l,E_g,F_g)$ as $\Lambda_g(\alpha,l)$.

 For any $x\in\Lambda_g(\alpha,l)$, we  denote the local unstable and local stable manifolds of $x$ from the previous lemma as $W^{u/s}_{\delta(\alpha,l)}(x)$, emphasizing that  the radius
$\delta(\alpha,l)$ of these local manifolds  depends only on the parameters $\alpha$ and $l$.

The following lemma is used to establish the existence of a Pesin block with positive measure for   any $c$-Gibbs $u$-state.

\begin{lem}\cite{CM}\label{yizhixing} Given $0<\alpha<\alpha_0$,  for any invariant measure $\mu$,  if $\mu$-almost every $x\in M$ satisfies
\begin{itemize}
\item $\lim_{n\to+\infty}\frac{1}{n}\log\| Df^{-n}|_{E(x)}\|<-\alpha_0$,
\item $\lim_{n\to+\infty}\frac{1}{n}\log\| Df^{n}|_{F(x)}\|<-\alpha_0,$ 
\end{itemize}
then for any $\varepsilon>0$, there exist $l\in\mathbb{N}$, a $C^1$-neighborhood $\mathcal{U}$ of $f$ and a neighborhood $\mathcal{V}$ of $\mu$ such that for every diffeomorphism $g\in\mathcal{U}$ and every $g$-invariant measure $\nu\in\mathcal{V}$, the following holds:
\[
\nu(\Lambda_g(\alpha,l))>1-\varepsilon.
\]
\end{lem}

Using Liao-Gan's shadowing lemma~\cite{39, 30}, Mi and Cao established the following lemma.  Hyperbolic periodic points in certain Pesin blocks accumulate on the support of the invariant measure restricted to a Pesin block of positive measure. The periodic points in these Pesin blocks are considered as candidates for the elements in our skeleton (see the proof of Lemma~\ref{A}).

\begin{lem}\cite{CM}\label{zhuqia}
For any $l\in\mathbb{N}$ and $\alpha>0$, there exists a constant $\eta(l,\alpha)>0$ such that the following holds: 	for any $f$-invariant measure $\mu$, if $\mu(\Lambda_f(\alpha,l))>0$, then for any point $x\in \supp(\mu|_{\Lambda_f(\alpha,l)})$ and any $\tilde{a}\le \eta(l,\alpha)$, the neghborhood $B(x,\tilde{a})$ contains a hyperbolic periodic point $p\in\Lambda(\frac{\alpha}{2},l)$ such that
\[
W^u_{\delta(\frac{\alpha}{2},l)}(p)\cap W^s_{\delta(\alpha,l)}(x)\ne\emptyset, W^s_{\delta(\frac{\alpha}{2},l)}(p)\cap W^u_{\delta(\alpha,l)}(x)\ne\emptyset.\]
\end{lem}

For any $x\in M$, the {\it stable set} of $x$ is defined as
$$
W^s(x,f):=\{y\in M:d(f^{k}(x),f^{k}(y))\to 0  ~~`\mbox{as k}\to+\infty\}.
$$
Similarly, the {\it unstable set} of $x$ is defined as $W^u(x,f):=W^s(x,f^{-1})$. It is clear from the definition that $f^{-1}(W^u(x,f))=W^u(f^{-1}(x),f)$. The sets $W^s(x,f)$ and $W^u(x,f)$ are uniquely determined by this definition.
Abdenur, Bonatti and Crovisier \cite{ABC} established the following proposition. In the $C^1$ setting, under the assumption of domination ($E\oplus_\succ F$),
unstable sets of almost every point are injectively immersed $C^1$-manifolds with respect to hyperbolic measures (see \cite[Proposition~8.9]{ABC} for further details). 

\begin{pro}\cite{ABC}\label{stablemanifold}
If the invariant measure $\mu$ satisfies the conditions stated in Lemma~\ref{yizhixing} and is ergodic, then for $\mu$-almost every $x$, $W^s(x,f)$ is an injectively immersed $C^1$-manifold of dimension $\dim(F)$, and $W^u(x,f)$ is an injectively immersed $C^1$-manifold of dimension $\dim(E)$.
\end{pro}

\subsection{Properties of Maps Factoring Over Anosov}

In this subsection, we discuss the properties of $c$-Gibbs $u$-states for maps that factor over Anosov systems.

Given a hyperbolic linear automorphism $A:\mathbb{T}^d\rightarrow \mathbb{T}^d$, 
if $f:M\to M$ factors over $A$, we say that $g$  {\it factors over the same Anosov as $f$} if $g$ also factors over $A$. Now,  suppose $f:M\to M$ factors over $A$ with the splitting $TM=E^{uu}\oplus_{\succ}E^{cu} \oplus_{\succ}E^{cs}$  in this subsection.

The following proposition from \cite{UVYY} lists fundamental properties of $c$-Gibbs $u$-states,  which are crucial for the proof of Lemma~\ref{linyu}.  By combining Lemma~\ref{zhuqia} with Theorem~\ref{uniform}, we establish that  the maximum number of $c$-$cu$-states with pairwise distinct supports are finite.  To prove Theorem~B, we primarily rely on Lemma~\ref{kakaka} and the item~\ref{a} of following proposition.

\begin{pro}\cite{UVYY}\label{main.pro}The following properties hold:
\begin{enumerate}
\item\label{c} The set of measures of maximal $u$-entropy, $Gibbs^u_c(f)$, is non-empty, convex, and compact.
\item\label{a} If $f_n$ factor over the same Anosov as $f$, where $n\in\mathbb{N}$ and  $f_n\rightarrow f$ in the $C^1$-topology, then
\[
\limsup_{n\rightarrow\infty}Gibbs^u_c(f_n)\subset Gibbs^u_c(f).
\]
\item\label{d} Almost every ergodic component of any $\mu\in Gibbs^u_c(f)$ is a $c$-Gibbs $u$-state.  Specifically, the set $\{\mu_P,P\in\P\}$ forms the ergodic decomposition of $\mu$, where $\hat{\mu}$ is the corresponding quotient measure in $\P$. This decomposition satisfies the following properties:
\begin{itemize}
\item the map $P\mapsto\mu_P(E)$ is measurable for every measurable set $E\subset M$;
\item $\mu(E)=\int\mu_P(E)d\hat{\mu}_P$ for every measurable set $E\subset M$;
\item  $\mu_P$ is an ergodic $c$-Gibbs $u$-state for $\hat{\mu}$-almost every $P\in\P$.
\end{itemize}
\item\label{f} Any accumulation point of the sequence \[\mu_n=\frac{1}{n}\sum_{j=0}^{n-1}f_*^j\nu^{uu}\] is a $c$-Gibbs $u$-state, where $\nu^{uu}$ is the reference measure on any strong-unstable plaque.
\end{enumerate}
\end{pro}

Based on above Proposition~\ref{main.pro}, we can obtain some following results, which will be useful in the forthcoming proofs.

\begin{lem}\label{xin} For any positive integer $m\in\mathbb{N}$, if $\mu\in EGibbs^u_c(f^m)$, then \[\mu_0=\frac{1}{m}\sum_{i=0}^{m-1}f_*^i\mu \in EGibbs^u_c(f),\]
where $f^i_*\mu=\mu(f^{-i})$ for every $i=0,1,...,m-1$.
\end{lem}
\begin{proof} By the definition of ergodic measures,   $\mu_0$ is also an ergodic measure of $f$. The $\pi^{-1}$-Markov partition $\mathcal{M}$ of $f$ can be selected as the  $\pi^{-1}$-Markov partition of $f^m$ using the same $\pi$. 
Thus the two  maps $f$ and $f^m$ have the same reference measures.

From the definition of $Gibbs^u_c(f^m)$, there exists some $j$ such that $\mu(\mathcal{M}_j)>0$. The Rokhlin disintegrations of $\mu|_{\mathcal{M}_j}$ along strong unstable plaques  coincide with the corresponding reference maesures almost everywhere. By the ergodicity of $\mu$, for $\mu$-almost every $x\in\mathcal{M}_j$ and $\nu^{uu}_{j,x}$-almost every $y\in \F^{uu}_j(x)$, we have:
$$\lim_{n\rightarrow+\infty}\frac{\sum_{0\le j\le n-1}\delta_{f^{mj}(y)}}{n}=\mu.$$

For any continuous function $\phi$ on $M$, we have: 
\[\lim_{n\rightarrow+\infty}\frac{1}{n}\sum_{i=0}^{n-1}\phi(f^{mi}(y))=\int\phi d\mu.\]
It follows that
\[\int\lim_{n\rightarrow+\infty}\frac{1}{n}\sum_{i=0}^{n-1}\phi(f^{mi}(y))d\nu^{uu}_{j,x}=\int(\int\phi d\mu )d\nu^{uu}_{j,x}=\int\phi d\mu.\]
By the compactness of $M$ and the Dominated Convergence Theorem, we conclude:
\[\int\lim_{n\rightarrow+\infty}\frac{1}{n}\sum_{i=0}^{n-1}\phi(f^{mi}(y))d\nu^{uu}_{j,x}=\lim_{n\rightarrow+\infty}\int\frac{1}{n}\sum_{i=0}^{n-1}\phi(f^{mi}(y))d\nu^{uu}_{j,x}=\lim_{n\rightarrow+\infty}\int\phi d \frac{1}{n}\sum_{j=0}^{n-1}f_*^{mj}\nu_{j,x}^{uu}.\]
To sum up, there exists $x_{\mu}$ such that: $$\mu=\lim_{n\rightarrow+\infty}\frac{1}{n}\sum_{j=0}^{n-1}f_*^{mj}\nu_{j,x_{\mu}}^{uu}.$$

Thus, we have:
$$\mu_0=\lim_{n\rightarrow+\infty}\frac{1}{mn}\sum_{j=0}^{nm-1}f_*^j\nu_{j,x_{\mu}}^{uu}.$$  By item ~\ref{f} of Proposition \ref{main.pro}, $\mu_0$ is a $c$-Gibbs $u$-state.
Therefore, $\mu_0$ is an ergodic $c$-Gibbs $u$-state.
\end{proof}

\begin{lem}\label{subset}
For any positive integer $m\in\mathbb{N}$, we have $$Gibbs^u_c(f)\subset Gibbs^u_c(f^m).$$
\end{lem}
\begin{proof}
Notice that the $\pi^{-1}$-Markov partition of $f$ is also a $\pi^{-1}$-Markov partition of $f^m$, with the same $\pi$. Moreover,  invariant measures of $f$ are also invariant measures of $f^m$. Therefore, by the definition of $Gibbs^u_c(f)$, it follows that any $c$-Gibbs $u$-state  for $f$ is also a $c$-Gibbs $u$-state  for $f^m$. Hence, $Gibbs^u_c(f)\subset Gibbs^u_c(f^m)$, which  completes the proof.
\end{proof}

\subsection{Properties of Maps Factoring Over Anosov with a $c$-Mostly Expanding  and $c$-Mostly Contracting Center}

In this subsection, $f$ factors over Anosov with the partially
hyperbolic splitting $TM = E^{uu} \oplus_{\succ} E^{cu} \oplus_{\succ} E^{cs}$ such that $f$ admits a $c$-mostly expanding  and $c$-mostly contracting center. For simplicity of notation, recall that $E=E^{uu}\oplus E^{cu}$ and $F=E^{cs}$.

The following lemmas, inspired by the works of  \cite[Proposition~5.4]{Yan} and \cite[Proposition~3.3]{ref7}, provide insight into the behavior of $c$-Gibbs $u$-states under such dynamics.
\begin{lem}\label{daijia}
There exist constants $N\in\mathbb{N}$ and $a>0$  such that for any $c$-Gibbs $u$-state $\mu\in Gibbs_c^u(f)$, the following inequality holds:
$$
\int \log \| Df^{-N}|_{E(x)}\| d \mu<-a.
$$
In addition, there exist constants $L\in\mathbb{N}$ and $b>0$  such that for any $c$-Gibbs $u$-state $\mu\in Gibbs_c^u(f)$, the  inequality 
$$
\int \log \| Df^{L}|_{F(x)}\| d \mu<-b
$$
also holds.
\end{lem}
\begin{proof}
 For any $c$-Gibbs $u$-state $\mu$,  by  Oseledets' theorem (see \cite{VO}),  smallest extremal Lyapunov exponents along subbundle $E$ coincide with $$\lim_{n\to+\infty} -\frac{1}{n}\log\| Df^{-n}|_{E(x)}\|$$ almost everywhere.  Thus, we have:
\[
\int\lim_{n\to+\infty} -\frac{1}{n}\log\| Df^{-n}|_{E(x)}\| d\mu>0.
\]
Then:
\[
\int\lim_{n\to+\infty} \frac{1}{n}\log\| Df^{-n}|_{E(x)}\| d\mu<0.
\]
By applying the Dominated Convergence Theorem, we can exchange the limit and the integral to obtain:
\[
\lim_{n\to+\infty} \frac{1}{n}\int \log\| Df^{-n}|_{E(x)}\| d\mu=\int\lim_{n\to+\infty} \frac{1}{n}\log\| Df^{-n}|_{E(x)}\| d\mu<0.
\]
This implies that there exist constants  $N_{\mu}$ and $a_{\mu}$ such that:
\[
\frac{1}{N_{\mu}}\int log\| Df^{-N_{\mu}}|_{E(x)}\| d\mu< -a_{\mu}<0.
\]
In other words:
$$
	\int \log\| Df^{-N_{\mu}}|_{E(x)}\| d\mu<- N_{\mu}a_{\mu}<0.
$$
Now, take a neighborhood $\mathcal{V}_{\mu}$ of $\mu$ in the space of probability measures on $M$, such that the above inequality holds for all $\nu\in \mathcal{V}_{\mu}$. By the compactness of $Gibb_c^u(f)$ (see item \ref{c} of Proposition ~\ref{main.pro}), there exists a finite open covering $\{\mathcal{V}_{\mu_i}\}_{i=1}^k$ of $Gibb_c^u(f)$, such that for any $\tilde{\nu}\in\mathcal{V}_{\mu_i}$, 
\begin{equation}\label{equa}
\int \log\| Df^{-N_{\mu_i}}|_{E(x)}\| d\tilde{\nu}<- N_{\mu_i}a_{\mu_i}<0.
\end{equation}

Let $N=\prod_{i=1}^k N_{\mu_i}$ and $a=\min\{a_{\mu_1},...,a_{\mu_k}\}$. For any $\tilde{\mu}\in Gibb_c^u(f)$, there exists some $1\le i_0\le k$ such that $\tilde{\mu}\in \mathcal{V}_{\mu_{i_0}}$.  By the chain rule,  along with the previously established inequality~\ref{equa} and our setting for $\mathcal{V}_{\mu_{i_0}}$, we  can deduce:
\[
\int \log \| Df^{-N}|_{E(x)}\| d \tilde{\mu}\le \frac{N}{N_{\mu_{i_0}}}\int \log \| Df^{-N_{\mu_{i_0}}}|_{E(x)}\| d \tilde{\mu}< -Na_{\mu_{i_0}}\le -a.
\]
Thus, the desired result follows. By employing a similar approach, we can derive the second part as well.
\end{proof}

\begin{lem}\label{alpha0}
There exists a constant $\alpha_0>0$ such that for any $c$-Gibbs $u$-state $\mu$, the following two conditions hold for $\mu$-almost every $x\in M$:
\begin{itemize}
\item $\lim_{n\rightarrow+\infty}\frac{1}{n}\log\| Df^{-n}|_{E(x)}\|<-\alpha_0$;
\item $\lim_{n\rightarrow+\infty}\frac{1}{n}\log\| Df^{n}|_{F(x)}\|<-\alpha_0$.
\end{itemize}
\end{lem}
\begin{proof}
We begin by considering the constant $N$ from lemma \ref{daijia}, and any ergodic $c$-Gibbs $u$-state $\mu\in EGibbs_c^u(f^{N})$. By  Lemma \ref{xin},
\[\mu_0=\frac{1}{N}\sum_{p=0}^{N-1}f_*^p\mu =\frac{1}{N}\sum_{i=0}^{-(N-1)}f_*^i\mu \in EGibbs^u_c(f),\]
where $f^{-N}_*\circ f_*^p\mu=f^{p-N}_*\mu=f_*^p\mu$ and $\mu=f^{-N}_*\mu$. Applying Lemma \ref{daijia}, we obtain:
$$ 
\int\log \| Df^{-N}|_{E(x)}\| d \mu_0=\int \log \| Df^{-N}|_{E(x)} \| d (\frac{1}{N}\sum_{i=0}^{-(N-1)}f_*^i\mu)<-a<0. 
$$
Thus, there exists $-(N-1)\le j\le 0$ such that:
\[
\int\log\| Df^{-N}|_{E(x)}\| d f_*^j\mu<-a.
\]
For any $x\in B(\mu,f^{-N})$, we have $f^{j}(x)\in B(f_*^j\mu, f^{-N})$. Since $\log\| Df^{-N}|_{E(x)}\|$ is a continuous function on $M$, by the definition of $weak^*$ convergence, we obtain:
\[
\lim_{n\rightarrow+\infty}\frac{1}{n}\sum_{l=0}^{n-1}\log\| Df^{-N}|_{E((f^{-N})^l(f^{j}(x)))}\|=\int \log \| Df^{-N}|_{E(x)}  \| d f_*^j\mu<-a.
\]
For simplicity, let $g=f^{-N}, x_j=f^{j}(x)$ just for the proof. The above formula can be used to write as:
\begin{equation}\label{budengshi1}
\lim_{n\rightarrow+\infty}\frac{1}{n}\sum_{l=0}^{n-1}\log\| Dg|_{E(g^l(x_j))}\|=\int \log \| Dg|_{E(x)}  \| d f_*^j\mu<-a.
\end{equation}

Let $k$ be any fixed positive integer. For every $l\in\mathbb{N}$, by  the chain rule:
\[
\| Dg^k|_{E(g^{kl}(x_j))}\|\le\prod_{i=0}^{k-1}\| Dg|_{E(g^{kl+i}(x_j))}\|.
\]
Thus:
\begin{equation}\label{budengshi2}
\frac{1}{n}\sum_{l=0}^{n-1}\log\| Dg^k|_{E(g^{kl}(x_j))}\|\le \frac{1}{n}\sum_{i=0}^{nk-1}\log\| Dg|_{E(g^i(x_j))}\|.
\end{equation}
Combining the  inequalities \ref{budengshi1} and \ref{budengshi2}, we have:
\begin{equation}\label{budengshi3}
\limsup_{n\rightarrow+\infty}\frac{1}{kn}\sum_{l=0}^{n-1}\log\| Dg^k|_{E(g^{kl}(x_j))}\|\le-a.
\end{equation}

Let $C=\max\{\log\| Df\|, \log\| Df^{-1}\|\}$.
Using chain rule  again, we obtain:
\begin{align*}
\begin{split}
\log\| Dg^k|_{E(g^{kl}(x_j))}\|&=\log\| Df^j\circ Dg^k\circ Df^{-j}|_{E(g^{kl}(x_j))}\|\\
&\ge \log\| Dg^k|_{E(g^{kl}(x))}\|-2C|j|\\
&\ge \log\| Dg^k|_{E(g^{kl}(x))}\|-2CN.
\end{split}
\end{align*}
Thus, we arrive at:
\begin{equation}\label{budengshi4}
\limsup_{n\rightarrow+\infty}\frac{1}{kn}\sum_{l=0}^{n-1}\log\| Dg^k|_{E(g^{kl}(x))}\|\le \limsup_{n\rightarrow+\infty}\frac{1}{kn}\sum_{l=0}^{n-1}(\log\| Dg^k|_{E(g^{kl}(x_j))}\|+2CN).
\end{equation}
Now, combining this with the previous inequalities~\ref{budengshi3}, we get: 
\[
\limsup_{n\rightarrow+\infty}\frac{1}{n}\sum_{l=0}^{n-1}\log\| Dg^k|_{E(g^{kl}(x))}\|\le \limsup_{n\rightarrow+\infty}\frac{1}{n}\sum_{l=0}^{n-1}\log\| Dg^k|_{E(g^{kl}(x_j))}\|\le(-ka+2CN).
\]
When $k$ is sufficiently large, $-ka+2CN$ becomes negative. By imposing $-ka+2CN<-a$, we derive $k>\frac{2CN}{a}+1$. Choosing 
\begin{equation}\label{kcN}
\frac{2CN}{a}+1<k<\frac{2CN}{a}+10,
\end{equation}
 we have:
\[
\limsup_{n\rightarrow+\infty}\frac{1}{n}\sum_{l=0}^{n-1}\log\| Dg^k|_{E(g^{kl}(x))}\|\le(-ka+2CN)<-a.
\]
Substituting $g=f^{-N}$ back, for $x\in B(\mu,f^{-N})$, we have: 
\[
\limsup_{n\rightarrow+\infty}\frac{1}{n}\sum_{l=0}^{n-1}\log\| Df^{-Nk}|_{E(f^{-Nkl}(x))}\|\le(-ka+2CN)<-a.
\]
Let $N_0=Nk$.  Therefore, for $\mu$-almost every $x\in M$(noting that $\mu(B(\mu,f^{-N}))=1$):
\[
\limsup_{n\rightarrow+\infty}\frac{1}{n}\sum_{i=0}^{n-1}\log\| Df^{-N_0}|_{E(f^{-iN_0}(x))}\|<-a.
\]
By the subadditive ergodic  theorem of Kingman, for $\mu$-almost every $x\in M$, the limit $$\lim_{n\rightarrow+\infty}\frac{1}{n}\log\| Df^{-n}|_{E(x)}\|$$ exists.
Thus, for $\mu$-almost every $x\in M$:
\begin{align*}
\begin{split}
	\lim_{n\rightarrow+\infty}\frac{1}{n}\log\| Df^{-n}|_{E(x)}\|&=\lim_{n\rightarrow+\infty}\frac{1}{nN_0}\log\| Df^{-nN_0}|_{E(x)}\|\\
&\le\limsup_{n\rightarrow+\infty}\frac{1}{nN_0}\sum_{i=0}^{n-1}\log\| Df^{-N_0}|_{E(f^{-iN_0}(x))}\|\\
&<-\frac{a}{N_0}\\
&<-\frac{a}{N(\frac{2CN}{a}+10)}(recall ~\ref{kcN}).
\end{split}
\end{align*}
It is evident that the value $-\frac{a}{N(\frac{2CN}{a}+10)}$ is independent of the choice of the ergodic $c$-Gibbs $u$-state $\mu$ of $f^N$.
Consequently, by setting $\alpha_0=\frac{a}{N(\frac{2CN}{a}+10)}$,  the measurable set defined by
$$
E^{-1}:=\{x:\limsup_{n\rightarrow+\infty}\frac{1}{n}\log\| Df^{-n}|_{E(x)}\|<-\alpha_0\}
$$
has full measure for any ergodic $c$-Gibbs $u$-state of $f^N$. 
By item~\ref{d} of Proposition~\ref{main.pro}, for any $c$-Gibbs $u$-state $\nu$ of $f^N$, $\nu(E^{-1})=1$. This implies that the smallest Lyapunov exponent on $E$ is not smaller than $\alpha$ for any $c$-Gibbs $u$-state $\nu$ of $f^N$.  Furthermore, by lemma \ref{subset}, $Gibbs_c^u(f)\subset Gibbs_c^u(f^N)$. Therefore, the set $E^{-1}$ has full measure for any  $c$-Gibbs $u$-state of $f$.  Combined with Kingman's subadditive ergodic theorem, we obtain the first inequality.

By the second part of Lemma~\ref{daijia},
a similar argument can be used to prove the corresponding inequality for the subbundle $F$.  By choosing $\alpha_0$ smaller if necessary,  we can ensure that the second inequality stated in this lemma also holds for almost every point.
\end{proof}

\begin{lem}\label{linyu} There exists a constant $\alpha>0$ such that for any $\varepsilon>0$, there exist a natural number $l\in\mathbb{N}$ and a $C^1$ neighborhood $\tilde{\mathcal{U}}$ of $f$ such that every $g\in\tilde{\mathcal{U}}$ factoring over the same Anosov as $f$, every measure $\mu\in Gibb^u_c(g)$ satisfies
\[
\mu(\Lambda_g(\alpha,l))>1-\varepsilon.
\]
\end{lem}
\begin{proof}
By combining Lemma~\ref{alpha0} and Lemma \ref{yizhixing}, and fixing $\varepsilon>0$ as in Lemma~\ref{yizhixing}, we conclude that for  any $\mu\in Gibb^u_c(f)$, there exist $l_{\mu}$, a neighborhood $\mathcal{V}_{\mu}$ of $\mu$, and a $C^1$ neighborhood $\mathcal{U}_{\mu}$ of $f$ such that for every $g\in\mathcal{U}_{\mu}$ and $g$-invariant measure $\nu\in\mathcal{V}_{\mu}$, we have
\[
\nu(\Lambda_{g}(\alpha, l_{\mu}))>1-\varepsilon.
\]
By the compactness of $Gibbs^u_c(f)$ (as noted in  item~\ref{c} of Proposition~ \ref{main.pro}), there exist finitely many $c$-Gibbs $u$-states $\mu_1,...,\mu_m$ such that $Gibb^u_c(f)\subset\bigcup_{1\le i\le m}\mathcal{V}_{\mu_i}$. Let $\tilde{\mathcal{U}}=\bigcap_{1\le i\le m}\mathcal{U}_{\mu_i}$.

 By shrinking $\tilde{\mathcal{U}}$ if necessary, we can ensure that $Gibb^u_c(g)\subset\bigcup_{1\le i\le m}\mathcal{V}_{\mu_i}$  for any $g\in\tilde{\mathcal{U}}$ that factors over the same Anosov  as $f$ (see item~\ref{a} of Proposition \ref{main.pro}). Define $l=l_{\mu_1}\cdot...\cdot l_{\mu_m}$. By the chain rule, we obtain: $$\Lambda_{g}(\alpha, l_{\mu_i})\subset \Lambda_g(\alpha,l), \mbox{ for each $i\in\{1,2,...,m\}$}.$$
As a result, for  any $\nu\in Gibbs^u_c(g)$, there exists $\mathcal{V}_{\mu_i}$ such that $\nu\in \mathcal{V}_{\mu_i}$.  Consequently, $$\nu(\Lambda_g(\alpha,l))\ge\nu(\Lambda_{g}(\alpha, l_{\mu_i}))>1-\varepsilon.$$
\end{proof}

\begin{lem}\label{openness}
Let $g\in\tilde{\U}$ be a diffeomorphism that factors over the same Anosov as $f$, where $\tilde{\U}$ is as described in Lemma~\ref{linyu}. Then $g$ has a $c$-mostly expanding  and $c$-mostly contracting center.
\end{lem}
\begin{proof}
From the definition of  Pesin blocks, for any $\gamma>0$,$k\in\mathbb{N}$ and any $x\in\Lambda_g(\gamma,k)$, we have:
	\[
	\limsup_{n\rightarrow+\infty}\frac{1}{n}\sum_{i=0}^{n-1}\log\| Dg^{-k}|_{E(g^{-ik}(x))}\|\le-\gamma k.\]
By Lemma~\ref{linyu}, there exists $\alpha>0$ such that for any  $c$-Gibbs $u$-state $\mu$ of $g$, we have  $\mu(\cup_{k\in\mathbb{N}}\Lambda_g(\alpha,k))=1$.  According to Kingman's subadditive ergodic theorem, for $\mu$-almost every $x$, the limit
$$
\lim_{n\rightarrow+\infty}\frac{1}{n}\log\| Dg^{-n}|_{E(x)}\|
$$
exists. Then, for $\mu$-almost every $x\in \cup_{k\in\mathbb{N}}\Lambda_g(\alpha,k)$, there exists some $k\in\mathbb{N}$ such that $x\in\Lambda_g(\alpha,k)$ and
\begin{align*}
\begin{split}
	\lim_{n\rightarrow+\infty}\frac{1}{n}\log\| Dg^{-n}|_{E(x)}\|&=\lim_{n\rightarrow+\infty}\frac{1}{nk}\log\| Dg^{-nk}|_{E(x)}\|\\
&\le\limsup_{n\rightarrow+\infty}\frac{1}{nk}\sum_{i=0}^{n-1}\log\| Dg^{-k}|_{E(g^{-ik}(x))}\|\\
&\le-\frac{\alpha k}{k}\\
&\le-\alpha.
\end{split}
\end{align*}
Consequently, the measurable set
	\[\{x:\limsup_{n\rightarrow+\infty}\frac{1}{n}\log\| Dg^{-n}|_{E(x)}\|\le-\alpha; \limsup_{n\rightarrow+\infty}\frac{1}{n}\log\| Dg^{n}|_{F(x)}\|\le-\alpha\}\]
	 has  full measure  for any  $c$-Gibbs $u$-state.
This implies that  the largest  Lyapunov exponent on $F$ is not bigger than $-\alpha$, and  the smallest Lyapunov exponent on $E$ is not smaller than $\alpha$. 
\end{proof}

We mainly use the following lemma to illustrate that for any $c$-$cu$-state,  local Pesin unstable manifold of  almost every point is contained in the support of the  $c$-$cu$-state.

\begin{thm}\label{uniform}
	 Let $\mu$ be an ergodic $c$-Gibbs $u$-state of $f$.  Then,  $\mu\in EG^{cu}(f)$ if and only if there is a full-measure set $\Gamma(\mu)^u\subset\cup_{k\in\mathbb{N}}\Lambda_f(\alpha,k)$, where $\alpha$ as in Lemma~\ref{linyu}, such that for any $x\in\Gamma(\mu)^u$
	\[W^u_{\loc}(x)\subset\supp(\mu).\]
	where $W^u_{\loc}(x)$ is the local unstable manifold of $x$, with dimension $dim(E^{uu}\oplus E^{cu})$, as in Lemma~\ref{kakaka}.
\end{thm}

Before we prove this theorem, we first observe the following fact.

\begin{lem}\label{suojinqu}
For any ergodic $c$-Gibbs $u$-state $\mu$,  there is a full-measure set $\Delta(\mu)\subset\cup_{k\in\mathbb{N}}\Lambda_f(\alpha,k)$, where $\alpha$ is as in Lemma~\ref{linyu}, such that for any $x\in \Delta(\mu)$, the following conditions hold:
\begin{itemize}
\item $W^u_\loc(x)\subset W^u(x,f);$
\item for any disk $D^u(x)\subset  W^u(x,f)$ containing $x$ and any $\tilde{\delta}>0$ satisfying $W^u_{\tilde{\delta}}(x)\subset W^u_{\loc}(x)$, if the diameter of $D^u(x)$  is less than $\frac{1}{2}\tilde{\delta}$ with respect to the metric on $W^u(x,f)$, then
 $$D^u(x)\subset W^u_{\tilde{\delta}}(x).$$
\end{itemize}
\end{lem}
\begin{proof}
By Lemma~\ref{linyu},  we know that $\mu(\cup_{k\in\mathbb{N}}\Lambda_f(\alpha,k))=1$. The first item follows from Lemma~\ref{kakaka}.

For the second item, by the  Proposition~\ref{stablemanifold}, $W^u(x,f)$ is an injectively immersed $C^1$-manifold of dimension $\dim(E)$. Since $D^u(x)\subset W^u(x,f)$ and $W^u_{\tilde{\delta}}(x)\subset W^u_{\loc}(x)\subset W^u(x,f)$, the metric on $W^u_{\tilde{\delta}}(x)$ is consistent with the metric on $W^u(x,f)$. Since  the diameter of $D^u(x)$  is less than $\frac{1}{2}\tilde{\delta}$ with respect to the metric on $W^u(x,f)$,  it directly follows that $D^u(x)\subset W^u_{\tilde{\delta}}(x)$.
\end{proof}

Now, we can prove Theorem~\ref{uniform}.

\begin{proof}[Proof of Theorem~\ref{uniform}]
	We begin by proving necessity.  Assume $\mu$ is a $c$-$cu$-state.
 By the definition of a $c$-$cu$-state, we have:
	\[\Gamma(\mu)=\cup_{n\in\mathbb{N}}\{x\in\Gamma(\mu),\delta(x)\ge\frac{1}{n}, W^u_{\delta(x)}(x)\subset\supp(\mu)\},\]
where $\Gamma(\mu)$ and $\delta(x)$ correspond to the definition associated with  $\mu$.
	Each set $$\{x\in\Gamma(\mu),\delta(x)\ge\frac{1}{n}, W^u_{\delta(x)}(x)\subset\supp(\mu)\}$$ is  measurable, as the condition $W^u_{\delta(x)}(x)\subset\supp(\mu)$ is naturally satisfied by the definition of the $c$-$cu$-state $\mu$.

	Given that $\mu(\Gamma(\mu))>0$,  the pigeonhole principle  guarantees the existence of an integer $n_0$ such that:
	\[	\mu(\{x\in\Gamma(\mu),\delta(x)\ge\frac{1}{n_0}, W^u_{\delta(x)}(x)\subset\supp(\mu)\})>0.\]
Define the set:
 $$\tilde{\Gamma}=\{x\in\Gamma(\mu),\delta(x)\ge\frac{1}{n_0}, W^u_{\delta(x)}(x)\subset\supp(\mu)\}.$$ 
Since $\Delta(\mu)$, which possesses the properties in Lemma~\ref{suojinqu}, is a full measure set, we can assume that $\tilde{\Gamma}\subset\Delta(\mu)$.

Next, consider $\Gamma^\prime=\cup_{k\in\mathbb{N}}\Lambda_f(\alpha,k)$, where $\alpha$ is as in Lemma~\ref{linyu}. By Lemma~\ref{linyu},  we have:
$$\mu(\Gamma^\prime)=\mu(\cup_{k\in\mathbb{N}}\Lambda_f(\alpha,k))=1.$$
Since $\mu(\tilde{\Gamma})>0$,    the Birkhoff Ergodic Theorem \cite{57} guarantees the existence of a full-measure subset  $\Gamma(\mu)^u\subset\Gamma^\prime$ such that for every $x\in \Gamma(\mu)^u$:
\begin{equation}\label{birkho}    
\lim_{n\rightarrow+\infty}\frac{\sum_{0\le j\le n-1}\delta_{f^{-j}(x)}}{n}(\tilde{\Gamma})=\mu(\tilde{\Gamma})>0.\end{equation}

 Now, let us fix any $\varepsilon\le\min\{\varepsilon_0,\frac{1}{n_0}\}$, where $\varepsilon_0$ is given in Lemma~\ref{kakaka}. For any $x\in \Gamma(\mu)^u$,  there exist infinitely many integers  $j\geq 1$ such that $f^{-j}(x)\in\tilde{\Gamma}$.   By Lemma~\ref{kakaka}, the diameter of $f^{-n}(W^u_\loc(x))$   becomes arbitrarily small for sufficiently large $n$, with respect to the metric on $f^{-n}(W^u_\loc(x))$.

Choosing $\varepsilon$ much smaller than $\frac{1}{n_0}$, we observe that the radius of the local unstable manifold for points in $\tilde{\Gamma}$ is uniformly  bounded below by $\frac{1}{n_0}$. Therefore, for  any $x\in \Gamma(\mu)^u$, we can select sufficiently large $j$ such  that $f^{-j}(x)\in\tilde{\Gamma}$ and diameter of $f^{-j}(W^u_\loc(x))$ is smaller than $\varepsilon$. Since $\tilde{\Gamma}\subset\Delta(\mu)$,  by applying Lemma~\ref{suojinqu}, we can then conclude that:
 $$f^{-j}(W^u_\loc(x))\subset W^u_{\frac{1}{n_0}}(f^{-j}(x)).$$
Since $W^u_{\frac{1}{n_0}}(f^{-j}(x))\subset\supp(\mu)$ and  considering the invariance of the support of  $\mu$, we conclude that there exists a full-measure subset  $\Gamma(\mu)^u\subset\Gamma^\prime$ such that $W^u_\loc(x)\subset\supp(\mu)$ for every $x\in\Gamma(\mu)^u$. 
    
    Now, we prove sufficiency. By Lemma~\ref{linyu} and Lemma~\ref{kakaka}, there is a positive-measure Pesin block  $\Lambda_f(\alpha,l)$ such that $\mu$-almost every $x\in\Lambda_f(\alpha,l)$, 
    \[W^u_{\delta(\alpha,l)}(x)= W^u_\loc(x)\subset\supp(\mu)\]
    where $\delta(\alpha,l)$ as in Lemma~\ref{kakaka}. Consider a constant function $\delta:\Lambda(\alpha,l)\to \R^+$   defined by $\delta\equiv\delta(\alpha,l)/2$. It follows directly from the definition that $\mu\in  EG^{cu}(f)$.
\end{proof}

\smallskip
\smallskip
\smallskip

\section{Proof of Theorem~A and B}

In this section, we provide the proof of the main theorems.  We fix $\alpha,l$ and any $\varepsilon\in(0,\frac{1}{4})$ as specified in Lemma~\ref{linyu}, and apply Lemmas \ref{zhuqia}, \ref{linyu},   \ref{kakaka} and Theorem~\ref{uniform} with these fixed parameters. Additionally, we simultaneously fix $\tilde{a}$ as given in Lemma~\ref{zhuqia}.

The Pesin block $\Lambda_g(\alpha,l)$ has weight for all $c$-Gibbs $u$-states of $g$ where $g\in\tilde{\U}$ factoring over Anosov as $f$ (see Lemma~\ref{linyu}).
The proof of the following lemma illustrates one approach for selecting a skeleton.

\begin{lem}\label{A}
	For any $f$ that factors over Anosov with a $c$-mostly expanding  and $c$-mostly contracting center, if $EG^{cu}(f)\neq\emptyset$, then there exists a skeleton $T(f)=\{q_1,...,q_\ell\}$ such that
	\begin{itemize}
		\item for each $\mu\in EG^{cu}(f)$, there exists at least one $q_i\in T(f)$ such that
		\[
		\supp(\mu)=\overline{W^u(\Orb(q_i))}, \overline{B(\mu)}\supset\overline{W^s(\Orb(q_i))};
		\]
		\item for each $q_j\in T(f)$, there exists at least one $\nu\in EG^{cu}(f)$ such that
		\[
		\supp(\nu)=\overline{W^u(\Orb(q_j))}, \overline{B(\nu)}\supset\overline{W^s(\Orb(q_j))}.
		\]
	\end{itemize}
\end{lem}
\begin{proof}
For any $\mu\in EG^{cu}(f)$,  Theorem~\ref{uniform} ensures that for   $\mu$-almost every $x\in\Lambda_f(\alpha,l)$ 
\[W^u_{\delta(\alpha,l)}(x)= W^u_\loc(x)\subset\supp(\mu).\]
Since $\Lambda_f(\alpha,l)$  has positive measure  for any $c$-$cu$-states(see Lemma~\ref{linyu}), it follows that $$\mu(\supp(\mu|_{\Lambda_f(\alpha,l)}))=\mu(\Lambda_f(\alpha,l))>0.$$ 
Next,  take  $x_\mu\in\supp(\mu|_{\Lambda_f(\alpha,l)})\cap B(\mu)\cap \Gamma(\mu)^u$, where $\Gamma(\mu)^u$ as in Theorem~\ref{uniform}.  By Lemma~\ref{zhuqia},  $B(x_\mu,\tilde{a})$ has some hyperbolic periodic point $p_{x_\mu}\in\Lambda_f(\frac{\alpha}{2},l)$ such that
\begin{equation}\label{periodic}
	W^u_{\delta(\alpha,l)}(x_\mu)\cap W^s_{\delta(\frac{\alpha}{2},l)}(p_{x_\mu})\ne\emptyset, W^s_{\delta(\alpha,l)}(x_\mu)\cap W^u_{\delta(\frac{\alpha}{2},l)}(p_{x_\mu})\ne\emptyset.
\end{equation}
 Since the local unstable manifolds and local stable manifolds of point in Pesin blocks are tangent to subbundles $E$ and $F$, respectively, it follows that points in $W^u_{\delta(\alpha,l)}(x_\mu)\cap W^s_{\delta(\frac{\alpha}{2},l)}(p_{x_\mu})$ and $W^s_{\delta(\alpha,l)}(x_\mu)\cap W^u_{\delta(\frac{\alpha}{2},l)}(p_{x_\mu})$ respectively are transverse intersections between stable manifolds and unstable manifolds.  Notice that $W^s(x_\mu)\subset B(\mu)$ and $W^u_{\delta(\alpha,l)}(x_\mu)\subset\supp(\mu)$.  By the inclination lemma, we have:
\begin{equation}\label{basin}
\overline{W^s(\Orb(p_{x_\mu}))}\subset \overline{B(\mu)}.
\end{equation}
and $\overline{W^u(\Orb(p_{x_\mu}))}\subset\supp(\mu).$

Now, take $y\in W^s_{\delta(\alpha,l)}(x_\mu)\cap W^u_{\delta(\frac{\alpha}{2},l)}(p_{x_\mu})$. Since $y\in W^s_{\delta(\alpha,l)}(x_\mu)\subset B(\mu)$ and $y\in W^u_{\delta(\frac{\alpha}{2},l)}(p_{x_\mu})$, it follows that:
\[\supp(\lim_{n\to+\infty}\frac{\sum_{0\le i\le n-1}\delta_{f^i(y)}}{n})\subset \overline{W^u(\Orb(p_{x_\mu}))}\subset\supp(\mu),\]
and \[\supp(\lim_{n\to+\infty}\frac{\sum_{0\le i\le n-1}\delta_{f^i(y)}}{n})=\supp(\mu).\] Thus, we conclude that:
\begin{equation}\label{supp}
\overline{W^u(\Orb(p_{x_\mu}))}=\supp(\mu).
\end{equation} 

Define  $T(f)^\prime=\{p_{x_\mu}:\mu\in EG^{cu}(f)\}$, which is the set  of  periodic points  corresponding to each $c$-$cu$-state, obtained
as described above. Notice that $T(f)^\prime\subset\Lambda_f(\frac{\alpha}{2},l)$.   Since every periodic point in $\Lambda_f(\frac{\alpha}{2},l)$ has   stable and unstable manifolds of uniform size,      there exists a maximal subset of $T(f)^\prime$, denoted by $T(f)=\{q_1,...,q_\ell\}$, satisfying the following properties:
\begin{itemize}
\item for any  $p_{x_\mu}\in T(f)^\prime$, there exists $q_j\in T(f)$  such that 
$W^u(p_{x_\mu})$ transversely intersects  $W^s(q_j)$, and $W^s(p_{x_\mu})$ transversely intersects  $W^u(q_j)$;
\item for any distinct $i\ne j\in\{1,2,...,\ell\}$, at least one of the following does not hold:  $W^u(q_j)\cap W^s(q_i)\neq\emptyset$ or $W^u(q_i)\cap W^s(q_j)\neq\emptyset$.
\end{itemize}
By the second item, $T(f)$ forms  a skeleton. By combining the first condition with the relationships established in~\ref{supp} and~\ref{basin}, it is straightforward to verify the remaining part of Lemma~\ref{A} using the inclination lemma.
\end{proof}

Now, we can proof  Theorem~A and B.
\begin{proof}[Proof of Theorem~A]
By Lemma~\ref{A}, we can directly obtain the results of Theorem~A.
\end{proof}
\begin{proof}[Proof of Theorem~B]
By Lemma~\ref{openness}, we can conclude the proof of the first item.

From Lemma~\ref{linyu}, we know that for  any $g\in\tilde{\mathcal{U}}$, which factors over the same Anosov as $f$, we have $\mu(\Lambda_g(\alpha,l))>1-\varepsilon$.

Furthermore,, notice that the elements of skeleton in Lemma~\ref{A} all belong to $\Lambda_g(\frac{\alpha}{2},l)$. Each point in $\Lambda_g(\frac{\alpha}{2},l)$ has the local  unstable manifold and local stable manifold with the uniform size, specifically  $\delta(\frac{\alpha}{2},l)$. We can thus collect all hyperbolic periodic points in $\Lambda_g(\frac{\alpha}{2},l)$. 

By combining  Lemma~\ref{kakaka} with the continuity of the sub-bundles along the dynamics, and by taking a small enough $C^1$-open subset $\mathcal{U}_f\subset\tilde{\mathcal{U}}$, we can ensure the existence of $\rho>0$ such that any $g\in\mathcal{U}_f$, $p_g,q_g\in\Lambda_g(\frac{\alpha}{2},l)$, if $d(p_g,q_g)\le\rho$, then
\[
W^u_{\delta(\frac{\alpha}{2},l)}(p_g)\cap W^s_{\delta(\frac{\alpha}{2},l)}(q_g)\ne\emptyset,W^u_{\delta(\frac{\alpha}{2},l)}(q_g)\cap W^s_{\delta(\frac{\alpha}{2},l)}(p_g)\ne\emptyset.
\]
Since $M$ is compact, there exists a finite covering of $M$ by $B(z_i,\frac{\rho}{2})$-balls  where $1\le i\le\tilde{\ell}$. Then each element of $T(g)$, the skeleton obtained by Lemma~\ref{A}, belongs to at most one of these $B(z_i,\frac{\rho}{2})$-balls.  

Therefore, the proof of Theorem~B is complete.
\end{proof}

\section{Qualitative Construction of Diffeomorphisms Factoring Over Anosov with a $c$-Mixed Center}\label{qual}

In this section, we construct diffeomorphisms that factor over  Anosov  with a $c$-mixed center, and we demonstrate the existence of uncountably many $c$-$cu$-states that share the same support.

Let  $h:M\rightarrow M$ be a partially hyperbolic diffeomorphism on the manifold $M$, with a tangent bundle splitting $TM=E^{uu}\oplus_{\succ} E^{cu} \oplus_{\succ} E^{cs}$. Suppose that $E^{cu}$ is either uniformly expanding, or trivial. 
Define a subset of $M$ as
$$
S=\{p_1,\cdots,p_k| \mbox{ each $p_i$ is a hyperbolic periodic point of stable index $dim E^{cs}$}\}.
$$
This set $S$ is called {\it a strong-skeleton} of $h$ if
\begin{itemize}
\item for any $x\in M$,  there exists $p_i\in T$ such that the unstable leaf tangent to $E^{uu}\oplus E^{cu}$ transversely intersects $W^s(\Orb(p_i))$ at some point;
\item $ W^u(p_j)\cap W^s(\Orb(p_i))=\emptyset$  when $i\not=j$.
\end{itemize}
Clearly, a strong-skeleton is also a skeleton in this setting. The skeleton $$
T=\{q_1,\cdots,q_m| \mbox{ each $q_i$ is a hyperbolic periodic point of stable index $dim E^{cs}$}\}.
$$
 is said to be {\it equivalent} to the strong-skeleton $S$ if $m\ge k$ and for
 each $q_i\in T$, there exists
a unique $p_{j(i)}\in S$ such that
\[ W^u(q_i)\cap W^s(\Orb(p_{j(i)}))\ne\emptyset, W^s(q_i)\cap W^u(\Orb(p_{j(i)}))\ne\emptyset.
\]
By domination, the stable and unstable manifolds of $q_i$ and $p_{j(i)}$  are  tangent to $E^{cs}$ and $E^{uu}\oplus E^{cu}$, respectively. Therefore, $W^u(q_i)$ transversely intersects $W^s(\Orb(p_{j(i)}))$ and $W^s(q_i)$ transversely intersects $W^u(\Orb(p_{j(i)}))$.
When $E^{cu}$ is uniformly expanding, by combining the definitions of the skeleton with the inclination lemma, it follows that $m = k$ and, for each $q_i \in T$, we have
\[
\overline{W^u(\Orb(q_i)}=\overline{W^u(\Orb(p_{j(i)})}; \overline{W^s(\Orb(q_i)}=\overline{W^s(\Orb(p_{j(i)})}.
\]

When constructing specific examples relevant to the research objectives, researchers often start with hyperbolic linear  automorphisms on the torus, modifying them  or placing them in the base space (e.g., derived from Anosov or partially hyperbolic skew-product).  In this paper, however, we adopt a qualitative approach,  as presented in the following proposition.
(In fact, the transitive Anosov diffeomorphism $T$ is positioned in the fiber of $f$. The final two sections offer a detailed explanation of this construction.  As we can observe from the examples in later sections, this construction is essentially a simplified version of the skew product of a skew product.
The following proposition provides the primary approach for constructing our related examples—particularly those with at least two $c$-$cu$-states that have distinct supports.)

\begin{pro}\label{pro.con} Let $A:\T^{d_1}\to\T^{d_1}$ be a hyperbolic linear automorphism, and assume that $N$ is a compact smooth Riemannian manifold. Let $f:N\to N$ be a partially hyperbolic diffeomorphism  that factors over $A$  with a partially hyperbolic splitting $E^{uu}\oplus_{\succ}E^{cs}$,  where $f$ admits a $c$-mostly contracting center along $E^{cs}$. Consider any transitive Anosov diffeomorphism $T:\T^{d_2}\to \T^{d_2}$ with the hyperbolic splitting $E^u_T\oplus E^s_T$ such that the product map $g=f\times T:N\times\T^{d_2}\to N\times\T^{d_2}$ admits a partially hyperbolic splitting \[TM=E^{uu}\oplus_{\succ}E^u_T\oplus_{\succ}( E^s_T\oplus E^{cs}),\]
where $M=N\times\T^{d_2}$.

Then, the diffeomorphism	$g$  factors over $A$ with a partially hyperbolic splitting \[TM=E^{uu}\oplus_{\succ}E^u_T\oplus_{\succ}( E^s_T\oplus E^{cs}),\] 
and satisfies the following conditions:
\begin{enumerate}
\item subbundle $E^{u}_T$ of $TM$ is $c$-mostly expanding, and subbundle $E^s_T\oplus E^{cs}$ of $TM$ is $c$-mostly contracting ; $E^{u}_T$  is as the “$E^{cu}$” in the definitions of strong-skeleton and the primary partially hyperbolic splitting of this study(the two definitions
are compatible);
\item there exists a strong-skeleton of $g$ such that
\begin{itemize}
\item the closure of the unstable manifolds through  the orbit of each periodic point in strong-skeleton supports  uncountable $c$-$cu$-state; 
\item  the maximum number of $c$-$cu$-states whose supports are pairwise distinct   is equal to the cardinality of the strong-skeleton of $g$;
\item the skeleton obained by Theorem~A is equivalent to the strong-skeleton;
\item there exists a strong-skeleton of $f$ such that the cardinality of the strong-skeleton of $f$ is equal to the cardinality of the strong-skeleton of $g$.
\end{itemize}
\end{enumerate}
Furthermore,  there exist a $C^1$-neighborhood $\mathcal{U}_f$ of $f$ and a $C^1$-neighborhood $\mathcal{U}_T$ of $T$ such that any  $\tilde{f}\in\mathcal{U}_f$ factoring over the same Anosov as $f$, and  $\tilde{T}\in\mathcal{U}_T$, $\tilde{g}=\tilde{f}\times\tilde{T}$ also factors over $A$ with a $c$-mostly expanding  and $c$-mostly contracting center.  For such $\tilde{g}$, there  exists a strong-skeleton of $\tilde{g}$ satisfying the corresponding properties as stated in the second item above.
\end{pro}

\begin{rk}
 Generally, when constructing diffeomorphisms with \( c \)-contracting centers, it is observed that \( \max\{\|Df|_{E^{cs}}\| : x \in N\} \leq \varepsilon + 1 \) for any sufficiently small \( \varepsilon \), and the action on the base, as introduced earlier, is an Anosov map. This allows us to select a suitable \( T \) to satisfy the assumptions of Proposition~\ref{pro.con}. We point out that in the sense of the skeleton (by Theorem~A) being equivalent to the strong-skeleton, the maximum number of $c$-$cu$-states with pairwise distinct supports can always be tracked as $g$ is perturbed to $\tilde{g}$. 
\end{rk}

The following lemma, coming from \cite[Theorem~C]{UVYY}, may simplify our proof.

\begin{lem}\label{zhuyao}\cite{UVYY}
	Let $f$ be as in the Proposition~\ref{pro.con}. Then there exist $\pi_1:N\to \T^{d_1}$, a strong-skeleton $\{p_i:i\in\{1,2,...,\ell\}\}$, and  exactly $\ell$ ergodic $c$-Gibbs $u$-states denoted by $\{\mu_{i}:{i\in\{1,2,...,\ell\}}\}$ characterized by the following properties:
		\begin{itemize}
            \item $f$ factors over Anosov via $\pi_{1}$;
			\item for each $i\in\{1,2,...,\ell\}$, $\supp(\mu_i)=\overline{W^{uu}(\Orb(p_i),f)}$;
            \item for $j\ne i$, $\overline{W^{uu}(\Orb(p_i),f)}\cap\overline{W^{uu}(\Orb(p_j),f)}=\emptyset$,
		\end{itemize}
where each $W^{uu}(\Orb(p_i),f)$ is  the strong unstable leaves through the orbit of  $p_i$.
\end{lem}

\smallskip

For convenience, we will directly use the corresponding symbols (such as $\pi_1,p_i,\mu_i$) from the above lemma~\ref{zhuyao} and define $T_i=\overline{W^{uu}(\Orb(p_i),f)}$. 
Next, we will  explain that $g$ factors over $A$.

By assumption, we know that $g$ admits a partially hyperbolic splitting
\[
TM=E^{uu}\oplus_{\succ}E^u_T\oplus_{\succ}( E^s_T\oplus E^{cs}).
\]
In this setting, $E^u_T$ is used as the subbundle $E^{cu}$ in the definition of factoring over Anosov in this paper.

Let $\E^c(g)$ be defined as:
$$
\{\E^c(x,f)\times\T^{d_2}: \mbox{ where $\E^c(f)$ is the unique $f$-invariant foliation tangent to $E^{cs}$ in system $(N,f)$ and $x\in N$}\}.
$$
 This definition implies that $\E^c(g)$ is  an $g$-invariant foliation, tangent to the subbundle $E^u_T\oplus E^s_T\oplus E^{cs}$ at every point in the system $(N\times\T^{d_2},g)$.
This implies that $g$ is dynamically coherent.

For any point $(y,x)\in M$ where $y\in N,x\in \T^{d_2}$, due to the uniqueness of the strong unstable foliation, the strong-unstable leaf of the partially hyperbolic diffeomorphism $g$ at the point $(y,x)$ is given by:
\begin{equation}\label{uu-home}
W^{uu}((y,x),g)=W^{uu}(y,f)\times\{x\},
\end{equation}
where $W^{uu}(y,f)$ denotes the strong-unstable leaf of the partially hyperbolic diffeomorphism $f$ at the point $y$.
Now,  define two maps $\pi$ and $\pi_{12}$ as follows:
\[
\pi_{12}: N\times\T^{d_2} \to N, \pi_{12}(y,x)=y;
\]
\[
\pi=\pi_1\circ\pi_{12}: M\to \T^{d_1},
\]
where $\pi_1$ is  as in Lemma~\ref{zhuyao}.
The commutative diagrams with the associated maps are as follows:
\[
\begin{CD}
N\times\T^{d_2} @>g>> N\times\T^{d_2}  \\
@VV\pi_{12}V     @VV\pi_{12}V\\
N    @>f>>  N\\
@VV\pi_{1}V     @VV\pi_{1}V\\
\T^{d_1}    @>A>>  \T^{d_1}
\end{CD};
\begin{CD}
N\times\T^{d_2} @>g>> N\times\T^{d_2} \\
@VV\pi=\pi_{1}\circ \pi_{12}V     @VV\pi=\pi_{1}\circ\pi_{12}V\\
\T^{d_1}    @>A>>  \T^{d_1}.
\end{CD}
\]
By combining equation~\ref{uu-home} and the first item of Lemma~\ref{zhuyao}, we see that the map $\pi$ maps each strong-unstable leaf of $g$ homeomorphically to an unstable leaf of $A$. By the commutative diagrams and the definition of  $\E^c(g)$, it follows that $\pi(\E^c((x,y),g))=W^s(\pi(x,y),A)$. 
Thus, we can conclude that the map $g$ factors over Anosov via $\pi$.

 To show that $g$ admits a $c$-mostly expanding  and $c$-mostly contracting center, by item~\ref{d} of Proposition~\ref{main.pro}, it suffices to prove that there exists some constant $\tilde{a}>0$ such that the measurable set
$$
\{(y,x):\limsup_{n\rightarrow+\infty}\frac{1}{n}\log\| Dg^{n}|_{E^{cs}(y,x)}\|<-\tilde{a}\}
$$
has full measure for any ergodic $c$-Gibbs $u$-state of $g$. (see Lemma~\ref{lem.factoroveranosov}). 

 Given that $T$ is a transitive Anosov diffeomorphism on torus, by   \cite[Lemma~1.2]{Frank},  there exists a fixed point of $T$, which we denote by $0$. 
 Let $S=\{O_i,O_i=(p_i,0), 1\le i\le\ell\}$,  where $0$ is a fixed point of $T$. Due to  the transitivity of $T$ and the result of Frank\cite{Frank},  we have:
\begin{equation}\label{clos}
\overline{W^u(\Orb((p_i,0),g))}=\overline{W^{uu}(\Orb(p_i),f)}\times \overline{W^u(0,T)}=T_i\times\T^{d_2};
\end{equation}
\begin{equation}\label{u-stable}
W^s(\Orb((p_i,0),g))=W^{s}(\Orb(p_i),f)\times W^s(0,T)
\end{equation}
By  uniqueness, the unstable leaf tangent to $E^{uu}\oplus E^u_T$ at the point $(y,x)$ is $W^{uu}(y,f)\times W^u(x,T)$.
Since $W^s(\Orb((p_i,0),g))$ tangent to $E^s_T\oplus E^{cs}$, combining Lemma~\ref{zhuyao} with the equation $$W^u(\Orb((p_i,0),g))=W^{uu}(\Orb(p_i),f)\times W^u(0,T),$$ it  follows that $S$ is a strong-skeleton. Before proving that $g$ admits a $c$-mostly expanding  and $c$-mostly contracting center, we  first need the following lemma.

\begin{lem}\label{contained}
For any ergodic $c$-Gibbs $u$-state $\nu$ of $g$, it holds that ${\pi_{12}}_*(\nu)=\mu_j$ for some $j\in\{1,2,3,...,\ell\}$, where $\mu_j$ as in Lemma~\ref{zhuyao}. Consequently, the support of $\nu$ is contained in $T_j\times\T^{d_2}$.
\end{lem}
\begin{proof}
First, notice that $f$ admits a $\pi_1^{-1}$-Markov partition,  and $g$ admits a $\pi^{-1}$-Markov partition. According to the exchange diagrams and the definition of $\pi_{12}$, we observe that the partition \[\{\pi_1^{-1}(\mathcal{R}_{1})\times\T^{d_2},\cdots,\pi_1^{-1}(\mathcal{R}_{k})\times\T^{d_2}\}\] forms a $\pi^{-1}$-Markov partition of $g$; Meanwhile, the partition \[\{\pi_1^{-1}(\mathcal{R}_{1}),\cdots,\pi_1^{-1}(\mathcal{R}_{k})\}\] is a $\pi^{-1}_1$-Markov partition of $f$, where   $\{\mathcal{R}_{1},\cdots,\mathcal{R}_{k}\}$ is a Markov partition of $A$.

By the definition of an ergodic $c$-Gibbs $u$-state and dominated convergence theorem, and following a similar argument as in the proof of Lemma~\ref{xin}, there exists a strong unstable plaque $W^{uu}_i((a,c),g)=W^{uu}_i(a,f)\times\{c\}$ with a reference measure $\nu^{uu}_i((a,c),g)$ for some point $(a,c)$ such that
\[
\lim_{n\to+\infty}\frac{1}{n}\sum_{j=0}^{n-1}g_*^j\nu^{uu}_i((a,c),g)=\nu.
\]

Next, we have the following relationships:
\begin{equation}\label{laiba}
\pi_{12}(W^{uu}_i((a,c),g))=\pi_{12}(W^{uu}_i(a,f)\times \{c\})=W^{uu}_i(a,f);
\end{equation}
\begin{equation}\label{qiyi}
{\pi_1}_*\circ{\pi_{12}}_*(\nu^{uu}_i((a,c),g))=\pi_*(\nu^{uu}_i((a,c),g))=\vol^u_{i,\pi(a,c)}=\vol^u_{i,\pi_1(a)};
\end{equation}
\begin{equation}\label{qier}
\pi_1(W^{uu}_i(a,f))=W^u_i(\pi_{1}(a),A).
\end{equation}
Combine above equlities \ref{laiba}, \ref{qiyi} and \ref{qier}, we conclude that ${\pi_{12}}_*(\nu^{uu}_i((a,c),g))$ is the reference measure on $W^{uu}_i(a,f)$.

By the continuity of $\pi_{12}$, we obtain:
\[
{\pi_{12}}_*(\lim_{n\to+\infty}\frac{1}{n}\sum_{j=0}^{n-1}g_*^j\nu^{uu}_i((a,c),g))={\pi_{12}}_*(\nu)=\lim_{n\to+\infty}\frac{1}{n}\sum_{j=0}^{n-1}{\pi_{12}}_*g_*^j\nu^{uu}_i((a,c),g).
\]
Using the commutativity of the exchange diagram, we have:
\[
\lim_{n\to+\infty}\frac{1}{n}\sum_{j=0}^{n-1}{\pi_{12}}_*g_*^j\nu^{uu}_i((a,c),g)=\lim_{n\to+\infty}\frac{1}{n}\sum_{j=0}^{n-1}f^j_*\circ{\pi_{12}}_*\nu^{uu}_i((a,c),g).
\]
Since ${\pi_{12}}_*(\nu^{uu}_i((a,c),g))$ is the reference measure on $W^{uu}_i(a,f)$, by item~\ref{f} of Proposition~\ref{main.pro},  this limit  is a $c$-Gibbs $u$-state of $f$. Given that $\nu$ is ergodic, it follows that ${\pi_{12}}_*(\nu)$ is also an ergodic $c$-Gibbs $u$-state.

Without loss of generality, assume that  ${\pi_{12}}_*(\nu)=\mu_j$ for some $j\in\{1,2,...,\ell\}$, as given in Lemma \ref{zhuyao}.
Observe that ${\pi_{12}}_*(\nu)(T_j)=\nu(T_j\times \T^{d_2})=\mu_j(T_j)=1$. This implies that
\begin{equation}\label{item1}
\supp(\nu)\subset T_j\times \T^{d_2}.
\end{equation}
\end{proof}

\begin{lem}\label{lem.factoroveranosov}There exists a constant $\tilde{a}>0$ such that
the measurable set
$$
\{(y,x):\limsup_{n\rightarrow+\infty}\frac{1}{n}\log\| Dg^{n}|_{E^{cs}(y,x)}\|<-\tilde{a},x\in\T^{d_2}\}
$$
has full measure for any ergodic $c$-Gibbs $u$-state of $g$. Consequently,  the subbundle $E^{cs}\oplus E^s_T$ is $c$-mostly contracting.
\end{lem}
\begin{proof}
 By directly applying Lemma~\ref{alpha0}, considering the case where $E^{cu}$ is trivial, there exists a constant $\tilde{a}>0$ such that the measurable set
$$
\{y: \limsup_{n\rightarrow+\infty}\frac{1}{n}\log\| Df^{n}|_{E^{cs}(y)}\|<-\tilde{a}\}
$$
has full measure for any  ergodic $c$-Gibbs $u$-state of $f$.

For any $n\in\mathbb{Z},(y,x)\in M$,  it holds that:
\begin{equation}\label{manzu1}
\log\|Dg^n|_{E^{cs}(y,x)}\|=\log\|Df^n|_{E^{cs}(y)}\|.
\end{equation}
Using the commutative diagrams and the equation~\ref{manzu1} above (Note that in both systems $(N,f),(N\times\T^{d_2},g)$, the corresponding \( E^{cs} \) is always invariant), we have
$$
\pi_{12}^{-1}(\{y: \limsup_{n\rightarrow+\infty}\frac{1}{n}\log\| Df^{n}|_{E^{cs}(y)}\|<-\tilde{a}\})=\{(y,x):\limsup_{n\rightarrow+\infty}\frac{1}{n}\log\| Dg^{n}|_{E^{cs}(y,x)}\|<-\tilde{a},x\in\T^{d_2}\}.
$$

Assuming  $\nu$ be an  ergodic $c$-Gibbs $u$-state as described in  Lemma~\ref{contained}. According to Lemma~\ref{contained},  we have ${\pi_{12}}_*(\nu)=\mu_j$ for some $j$. Thus,
$$
{\pi_{12}}_*(\nu)(\{y: \limsup_{n\rightarrow+\infty}\frac{1}{n}\log\| Df^{n}|_{E^{cs}(y)}\|<-\tilde{a}\})=\mu_j(\{y: \limsup_{n\rightarrow+\infty}\frac{1}{n}\log\| Df^{n}|_{E^{cs}(y)}\|<-\tilde{a}\})=1
$$
This implies that
$$
\nu(\{(y,x):\limsup_{n\rightarrow+\infty}\frac{1}{n}\log\| Dg^{n}|_{E^{cs}(y,x)}\|<-\tilde{a},x\in\T^{d_2}\})=1.
$$
Due to  arbitrariness of ergodic $c$-Gibbs $u$-state $\nu$,  we conclude  from item \ref{d} of Proposition \ref{main.pro} that the largest Lyapunov exponent along $E^{cs}$ is  less than $-\tilde{a}$ almost everywhere for any  $c$-Gibbs $u$-state of $g$.
\end{proof}

\begin{rk}\label{mostlycontra}
We point out that if $f$  satisfies that $f$ is $C^{1+}$-partially hyperbolic diffeomorphism on the compact smooth Riemannian manifold $N$, with the splitting $E^{uu}\oplus_{\succ} E^{cs}$, where $E^{cs}$ is mostly contracting(in the sense of Gibbs $u$-states), then
we can  select a suitable transitive Anosov $C^{1+}$-diffeomorphism $T$ with hyperbolic splitting $E^{u}_T\oplus E^{s}_T$ such that the product map $f \times T : N \times \T^{d_2} \to N \times \T^{d_2}$ admits a partially hyperbolic splitting: \[TM=E^{uu}\oplus_{\succ}E^u_T\oplus_{\succ}( E^s_T\oplus E^{cs}).\]
Since Gibbs $u$-states have the properties similar to those in Proposition~\ref{main.pro} (particularly item~\ref{f}), and   the result in \cite[Theorem~A]{DVY} regarding physical measures is similar to Lemma~\ref{zhuyao}, one can show that $f\times T\in PH_{EC}^{1+}(M)$ by replacing the reference measure with the Lebesgue measure on the strong unstable leaf.  Here, $PH_{EC}^{1+}(M)$ is defined by  Mi and  Cao in \cite{CM}.
By the way, it is clear  that the cardinality of  strong-skeleton of $f\times T$ is equal to the cardinality of  strong-skeleton of $f$.
\end{rk}

Next, we will demonstrate that there exists an uncountable number of $c$-$cu$-states supported on each set $T_i \times \T^{d_2}$. Before doing so, we require the following lemma:

\begin{lem}\label{ifonlyh}
A measure $\mu$ is a $c$-$cu$-state of $g$ if and only if $\mu$ is an ergodic $c$-Gibbs $u$-state  of $g$, and  $\supp(\mu)$ is the closure of the unstable manifolds of dimension $\dim(E^{uu}\oplus E^u)$ through the orbit of some periodic point in the strong-skeleton.
\end{lem}
\begin{proof}
The proof of sufficiency: 
Assume that $\mu$ is  an ergodic $c$-Gibbs $u$-state  of $g$ with $\supp(\mu) = \overline{W^u(\Orb(p))}$, where $p$ belongs to the strong-skeleton.
Consider a constant function $$\delta:\overline{W^u(\Orb(p))}\cap\Lambda_g(\alpha,l)\to \R^+,$$  defined as $\delta\equiv\delta(\alpha,l)$,  and $\Lambda_g(\alpha,l)$ as in Lemma~\ref{linyu} with positive measure (Note that $g$ already has a $c$-mixed center). 
Since $E^{uu}\oplus E^u$ is uniformly expanding, the continuity of the unstable manifolds tangent to $E^{uu}\oplus E^u$ implies the following: for any $x\in\overline{W^u(\Orb(p))}\cap\Lambda_g(\alpha,l)$, the unstable manifold $W^u(x)$ tangent to $E^{uu}\oplus E^u$ is contained in  $\overline{W^u(\Orb(p))}$. The latter is the support of $\mu$.
By combining this with the definition of a 
 $c$-$cu$-state, we can end the proof.

The proof of necessity: By Theorem~A, there exists some periodic point in the skeleton, which we denote by $p_\mu$, such that $\supp(\mu)=\overline{W^u(\Orb(p_\mu))}$. By Lemma~\ref{contained}, $\overline{W^u(\Orb(p_\mu))}\subset \overline{W^u(\Orb(O_j))}$ for some $j\in\{1,2,...,\ell\}$.
Since ${E^{uu}\oplus E^{u}}$ is uniformly expanding. Consequently,  $W^s(p_\mu)$ must transversely intersect $W^u(\Orb(O_j))$.
By the definition of strong-skeleton, $W^u(p_\mu)$ has to transversely intersect $W^s(\Orb(O_i))$ for some $i\in\{1,2,...,\ell\}$. By the inclination lemma, $W^u(\Orb(O_j))$ transversely intersects $W^s(\Orb(O_i))$. By the definition of strong-skeleton again, it follows that $i=j$. Applying the inclination lemma once again, we deduce that  $\overline{W^u(\Orb(p_\mu))}=\overline{W^u(\Orb(O_j))}$. Thus, $\supp(\mu)=\overline{W^u(\Orb(O_j))}$.
\end{proof}

\begin{lem}\label{more cu-state}
For each $T_i \times \T^{d_2}=\overline{W^u(\Orb((p_i,0),g))}$, there exist uncountably many  $c$-$cu$-states whose support is  exactly $T_i\times \T^{d_2}$.
\end{lem}
\begin{proof}
Consider an ergodic measure $\tilde{\nu}$ on system $(\T^{d_2}, T)$ such that $\supp(\tilde{\nu})=\T^{d_2}$. There exists a measurable set $\Gamma=B(\tilde{\nu})$ with $\tilde{\nu}(\Gamma)=1$, where $B(\tilde{\nu})$ is the basin of $\tilde{\nu}$ in system $(\T^{d_2}, T)$.

We construct a product measure $\mu_i\times \tilde{\nu}$ on $T_i\times\T^{d_2}$. 
Notice that $\mu_{i}\times \tilde{\nu}(T_i\times\Gamma)=1$ and that $\mu_{i}\times \tilde{\nu}$ is an invariant measure of $g$. 
By transitivity of the disintegration(see \cite[Exercise~5.2.1]{VO} or \cite{VARA}), it is easy to obtain that $\mu_{i}\times \tilde{\nu}$ is a $c$-Gibbs $u$-state. By ergodic decomposition theorem, almost every ergodic component of $\mu_{i}\times \tilde{\nu}$ has full measure on $T_i\times\Gamma$. 
By item~\ref{d} of Proposition~\ref{main.pro},  there exists an ergodic component $\tilde{\mu}$ such that $\tilde{\mu}$ is a ergodic $c$-Gibbs $u$-state and satisfies $\tilde{\mu}(T_i\times\Gamma)=1$.

The family of strong unstable plaques associated with points in $T_i\times\Gamma$  is  given by
$$
\{W^{uu}_j(x)\times\{y\},x\in T_i, y\in\Gamma, j\in\{1,2,...,k\}.\},
$$ where $k$ denotes the cardinality of the $\pi^{-1}$-Markov partition, and $W^{uu}_j(x)\times\{y\}$ is contained in corresponding $\mathcal{M}_j$.  Each strong unstable plaque of a point in $T_i\times\Gamma$  remains contained in $T_i\times\Gamma$. 

Since $\tilde{\mu}(T_i\times\Gamma)=1$, according to the definition of  a $c$-Gibbs $u$-state and  ergodicity of  $\tilde{\mu}$ (following a similar argument as in Lemma~\ref{xin}),  there exists a strong unstable plaque $W^{uu}_i(z)\times\{w\}$ with the reference measure $\nu^{uu}_i(z,w)$ such that
\begin{equation}\label{limitu}
\lim_{n\to+\infty}\frac{1}{n}\sum_{j=0}^{n-1}(f\times T)^j_*\nu^{uu}_i(z,w)=\tilde{\mu}, W^{uu}_i(z)\times\{w\}\subset \supp(\tilde{\mu})\cap (T_i\times\Gamma).
\end{equation}
 Note that $\nu^{uu}_i(z,w)=\nu^{uu}_i(z)\times\delta_w$, where $\nu^{uu}_i(z)$ is the reference measure on $W^{uu}_i(z)$ for the system $(M,f)$.

Consider the projection map:
\[
\pi_2: T_i\times\T^{d_2}\to \T^{d_2}, \pi_2(x,y)=y.
\]
We have the following commutative diagram:
\[
\begin{CD}
T_i\times\T^{d_2} @>f\times T>> T_i\times\T^{d_2}  \\
@VV\pi_2V     @VV\pi_2V\\
\T^{d_2}    @>T>>  \T^{d_2}
\end{CD}
\]
Applying $\pi_2$ to the limit in equation \eqref{limitu}, we get:
\[
{\pi_2}_*(\lim_{n\to+\infty}\frac{1}{n}\sum_{j=0}^{n-1}(f\times T)^j_*\nu^{uu}_i(z,w))={\pi_2}_*(\tilde{\mu})
\]
This simplifies to:
\begin{align*}
\begin{split}
{\pi_2}_*(\lim_{n\to+\infty}\frac{1}{n}\sum_{j=0}^{n-1}(f\times T)^j_*\nu^{uu}_i(z,w))&={\pi_2}_*(\lim_{n\to+\infty}\frac{1}{n}\sum_{j=0}^{n-1}(f\times T)^j_*(\nu^{uu}_i(z)\times\delta_w))\\
&=\lim_{n\to+\infty}\frac{1}{n}\sum_{j=0}^{n-1}{\pi_2}_*(f\times T)^j_*(\nu^{uu}_i(z)\times\delta_w)\\
&=\lim_{n\to+\infty}\frac{1}{n}\sum_{j=0}^{n-1}T^j_*{\pi_2}_*(\nu^{uu}_i(z)\times\delta_w)\\
&=\lim_{n\to+\infty}\frac{1}{n}\sum_{j=0}^{n-1}T^j_*\delta_w\\
&=\tilde{\nu}.(w\in B(\tilde{\nu},T))
\end{split}
\end{align*}

Thus, we have shown that $\pi_{2*}(\tilde{\mu}) = \tilde{\nu}$. Next, we establish that $\tilde{\mu}$ is indeed a $c$-$cu$-state. Before proceeding, we introduce the following claim. 

\begin{claim}\label{densev}
 For any $\varepsilon>0$, there exists an integer $N$ such that for all $n\ge N$ $f^n(W^{uu}_i(z))$ is $\varepsilon$-dense in $T_i$.
\end{claim}
\begin{proof}[Proof of Claim]
By the relation~\ref{limitu}, we 
observe that $\nu^{uu}_i(z)$ is the reference measure on the strong unstable plaque $W^{uu}_i(z)$ of the system $(M,f)$, where $W^{uu}_i(z)\subset T_i$. Since the invariant set $T_i$ supports a unique  ergodic $c$-Gibbs $u$-state, which is $\mu_i$, we can apply items~\ref{d} and \ref{f} of Proposition~\ref{main.pro} to obtain:
\[
\lim_{n\to+\infty}\frac{1}{n}\sum_{j=0}^{n-1}f^j_*(\nu^{uu}_i(z))=\mu_{i}.
\] 
Now,   consider a closed neighborhood $C$ of $p_i$ such that for any point $q\in C$, the strong unstable manifold $W^{uu}(q)$ transversely intersects  $W^s_{\loc}(p_i)$ at some point. Recall that 
$$
T_i=\overline{W^{uu}(\Orb(p_i),f)}=\supp(\mu_i).
$$
 Since $p_i\in\supp(\mu_i)$, we  conclude that:
\[\liminf_{n\to+\infty}\frac{1}{n}\sum_{j=0}^{n-1}f^j_*(\nu^{uu}_i(z))(C)\ge\mu_{i}(C)>0.\]
This inequality implies that for sufficiently large  $n$,  $f^n(W^{uu}_i(z))$ must transversely intersect $W^s_{\loc}(p_i)$ at some point.

By the compactness of $\overline{W^{uu}(\Orb(p_i))}$, the set $W^{uu}_R(\Orb(p_i))$ is $\varepsilon$-dense in $\overline{W^{uu}(\Orb(p_i))}$ for   sufficiently large $R$. Since $T_i$ is invariant under $f$, we have $f^n(W^{uu}_i(z))\subset\overline{W^{uu}(\Orb(p_i),f)}$ for all $n$.   Applying inclination lemma, we conclude that  $f^n(W^{uu}_i(z))$ is $\varepsilon$-dense in $\overline{W^{uu}(\Orb(p_i))}$(with respect to the subspace topology) for sufficiently large $n$.
\end{proof}

Next, we aim to prove that $\tilde{\mu}$ is a $c$-$cu$-state. To achieve this, it suffices to show that $\supp(\tilde{\mu})=T_i\times\T^{d_2}$, as established by Lemma~\ref{ifonlyh}.

For any point $(x,y)\in T_i\times\T^{d_2}$ and any closed neighborhood $U\times V$ of $(x,y)$, we have:  $${\pi_2}_*(\tilde{\mu})(V)=\tilde{\nu}(V)>0.$$ 
Since $\tilde{\mu}(T_i\times\T^{d_2})=1$, there exist  positive constants $\delta_1>\delta_2>0$ such that  $$\tilde{\mu}(T_i\times V)\ge\delta_1>\delta_2>0.$$ By relation \ref{limitu},  for sufficiently large $n$, we have:
\[
\frac{1}{n}\sum_{j=0}^{n-1}(f\times T)^j_*\nu^{uu}_i(z,w)(T_i\times V)=\frac{1}{n}\sum_{j=0}^{n-1}(f\times T)^j_*(\nu^{uu}_i(z)\times\delta_w)(T_i\times V)\ge\delta_2>0.
\]
Thus, there are infinitely many $j$ such that $(f\times T)^j(W^{uu}_i(z)\times\{w\})$ intersects $T_i\times V$.

By selecting $\varepsilon$ much smaller than the diameter of $U$ (where $\varepsilon$ as in above Claim~\ref{densev}), and then applying  Claim~\ref{densev}, we deduce that $(f\times T)^j(W^{uu}_i(z)\times\{w\})$ intersects $U\times V$ for some $j\ge N$.  By the invariance of the support and relation~\ref{limitu}, $(f\times T)^j(W^{uu}_i(z)\times\{w\})\subset\supp(\tilde{\mu})$.
Since $$(f\times T)^j(W^{uu}_i(z)\times\{w\})\cap (U\times V)\neq\emptyset$$ and given that $\supp(\tilde{\mu})$ is  closed, we can conclude that $\supp(\tilde{\mu})=T_i\times\T^{d_2}$.   Next, we establish that there are uncountably many $c$-$cu$-states, all with support equal to $T_i\times\T^{d_2}$.

It is well known that there are uncountably many ergodic measures of $T$ with support on $\T^{d_2}$. From our earlier discussion, $\tilde{\mu}$ is an ergodic component of  $\mu_{i}\times\tilde{\nu}$ that satisfies ${\pi_2}_*(\tilde{\mu})=\tilde{\nu}$ and  is a $c$-$cu$-state.  Since  different ergodic measures of $T$ with full support on $\T^{d_2}$ induce distinct $c$-$cu$-states, it follows that  there exist uncountable $c$-$cu$-state, all supported on  $T_i\times \T^{d_2}$.
\end{proof}

\begin{proof}[Proof of Proposition~\ref{pro.con}]
By Lemma~\ref{lem.factoroveranosov}, $g$ has a $c$-mostly expanding  and $c$-mostly contracting center, thereby establishing the first part of the proposition.

To prove the second part,   consider the set $S=\{O_i,O_i=(p_i,0), 1\le i\le\ell\}$, which  is a strong-skeleton. Notice that the closures of  the unstable manifolds passing through the orbit of distinct periodic points in the strong-skeleton are pairwise disjoint by Lemma~\ref{zhuyao}.
Using  Lemma~\ref{ifonlyh} and Lemma~\ref{more cu-state}, we establish the first two items of the second part of the proposition. The proof of Lemma~\ref{ifonlyh} also provides the necessary details to complete the third item.
 Since the strong-skeleton of the system $(g,M)$ can be constructed from the strong-skeleton of the system $(f,N)$, it follows that the fourth item of the second part holds ture.

	The property of having a $c$-mostly contracting is $C^1$-open among partially hyperbolic diffeomorphisms that factor over the same Anosov. This means there exists a $C^1$-neighborhood $\mathcal{U}_f$ of $f$ such that any $\tilde{f}\in\mathcal{U}_f$ factoring over the same Anosov as $f$ has a $c$-mostly contracting center along $E^{cs}_{\tilde{f}}$. 
Since transitive Anosov diffeomorphisms on the torus are structurally stable, there exists a  $C^1$-neighborhood $\mathcal{U}_T$ such that any $\tilde{T}\in \mathcal{U}_T$ remains a transitive Anosov diffeomorphism. 
The property of having partially hyperbolic splittings is also open among diffeomorphisms.
By appropriately shrinking the neighborhoods $\mathcal{U}_f$ and $\mathcal{U}_T$, we can guarantee that $\tilde{g}$ satisfies the same conditions  as $g$.
\end{proof}

\subsection{Examples of Our Results and the Applications of Proposition~\ref{pro.con}}\label{nonapp}

Dolgopyat, Viana, and Yang \cite{DVY} constructed a family of partially hyperbolic $C^2$-diffeomorphisms on the product manifold $\T^2 \times S^2$, where  $S^2$ denotes a sphere. These diffeomorphisms admit partially hyperbolic splittings of the form $$E^{uu}\oplus_{\succ}E^c\oplus_{\succ}E^{ss},$$ where $E^c=TS^2$ is an invariant subbundle of tangent bundle $T(\T^2\times S^2)$. (For convenience, we  will use notation“$E^{uu}\oplus_{\succ}E^c\oplus_{\succ}E^{ss}$” to denote the partially hyperbolic splittings associated with this family of diffeomorphisms).

  For any positive integer $n\in\mathbb{N}$, there exist partially hyperbolic  diffeomorphisms on $\T^2\times S^2$ with a strong-skeleton of cardinality  $n$, where each periodic point in the strong-skeleton is fixed point. This is achieved by  choosing $E^c\oplus E^{ss}$ as $E^{cs}$ in the definition of strong-skeleton. The stable manifold  of each fixed point in the strong-skeleton is dense in the entire space. More precisely, the unstable manifold of each fixed point in the strong-skeleton is tightly adjacent to the stable manifolds of all other fixed points in the strong-skeleton. 
Locally, to clarify this using local coordinates, the unstable manifold of each fixed point in strong-skeleton contains points of the form $(x,0,0,0)$, where $x\in(-a,a)$. The stable manifold of any other fixed point in strong-skeleton contains points of the form $(0,y,z,w)$, where $y\in(-a,a)$ and $z,w\in(-a,0)\cup(0,a)$.
This implies that  by composing with some locally supported $C^\infty$-perturbations on a neighborhood of the unstable manifold of a fixed point in skeleton, the cardinality of the strong-skeleton of the resulting partially hyperbolic diffeomorphism can be correspondingly reduced.
 Furthermore, it has been established in \cite{DVY} that different strong-skeletons are mutually equivalent. This result further implies that different strong-skeletons are mutually equivalent in partially hyperbolic splittings of the form $$E^{uu}\oplus_{\succ}E^{cu} \oplus_{\succ}E^{cs}$$ when $E^{cu}$ is uniformly expanding(see the definition of strong-skeletons).

\subsubsection{Some Examples}

For any positive integer $n$, it has been shown  in \cite{LiZhang} that we can choose an appropriate  partially hyperbolic diffeomorphism $\hat{f}:\T^2 \times S^2\to \T^2 \times S^2$ that factors over Anosov, such that the center-stable subbundle $E^c\oplus E^{ss}$ is $c$-mostly contracting and the cardinality of the strong-skeleton is $n$(the choices and constructions of the family of diffeomorphisms are compatible).

 Furthermore, we can carefully select a suitable transitive Anosov diffeomorphism $\hat{T}:\T^2\to\T^2$ with the hyperbolic splitting $E^{u}\oplus E^{s}$ such that the product map 
$$
\hat{f}\times\hat{T}:\T^2\times S^2\times\T^2\to \T^2\times S^2\times\T^2
$$ admits the the partially hyperbolic splitting: $$TM=E^{uu}\oplus_{\succ}E^{u}\oplus_{\succ}( E^c\oplus E^s \oplus E^{ss}),$$
where $M=\T^2\times S^2\times\T^2$.

 Thus, $\hat{f}\times\hat{T}$ satisfies the general assumption of Proposition~\ref{pro.con}. In particular, it has been shown in \cite{LiZhang} that there exists a  $C^1$-neighborhood $\U_{\hat{f}}$ such that any  diffeomorphism in $\U_{\hat{f}}$ factors over the same Anosov as $\hat{f}$ with a $c$-mostly contracting center.  

The example discussed above allows us to determine the precise number of $c$-$cu$-states with pairwise distinct (or disjoint) supports,  which corresponds to the cardinality of the strong-skeleton.

Ures, Viana, F. Yang and J. Yang proved in \cite{UVYY} that diffeomorphisms {\it derived from Anosov} factor over Anosov with $c$-mostly contracting centers (see \cite{UVYY} for more information).  This terminology originates from Smale \cite{Smale} in the study of 2-dimensional maps. Ma\~{n}\'{e}
 \cite{Man} was the first to investigate partially hyperbolic diffeomorphisms derived from Anosov. Ures, Viana, F. Yang and J. Yang also explained  partially volume expanding topological solenoids (coming from \cite{BLY,SMMJ}) factor over Anosov and have $c$-mostly contracting center, restricted to the maximal invariant set. 
Therefore, these diffeomorphisms all can be considered as the $f$ in Proposition~\ref{pro.con}.

We point out that if $T$ is a non-transitive Anosov diffeomorphism in Proposition~\ref{pro.con}, a similar argument shows that the corresponding $g$ factors over $A$ with a $c$-mixed center. However, determining the precise maximum number of  $c$-$cu$-states with pairwise distinct supports is rather cumbersome. Additionally, whether a non-transitive Anosov diffeomorphism exists remains an open question.

\subsubsection{Application of Proposition~\ref{pro.con}}

Recall that {\it Gibbs $u$-states} are invariant measures whose conditional measures  along the leaves of the strong-unstable foliation  (the unique foliation 
tangent to the strong unstable subbundle) are absolutely continuous with respect to the  corresponding Lebesgue measure. In the $C^{1+}$-setting, the existence of Gibbs $u$-states is guaranteed for partially hyperbolic systems. 
 Gibbs $u$-states in these systems are closely related to physical measures, and the existence of  physical measures often requires the partially hyperbolic diffeomorphism to be at least $C^{1+}$.
Thus, the partially hyperbolic systems discussed in this subsubsection are considered in the $C^{1+}$-setting. Next, we outline the applications of Proposition~\ref{pro.con}.

In  a partially hyperbolic  splitting $TM=E^{uu}\oplus_{\succ}E^{cs}$,  the subbundle $E^{cs}$ is said to be  {\it mostly contracting} if every Gibbs $u$-state has only negative Lyapunov exponents along $E^{cs}$. The focus of the research in \cite{BV,DVY} is on partially hyperbolic diffeomorphisms characterized by such splittings with contracting centers.  Dolgopyat, Viana and Yang obtained in \cite{DVY} that the number of physical measures coincides with the cardinality of any strong-skeleton.

For a more complex partially hyperbolic splitting of the form $E^{uu}\oplus_{\succ}E^{cu}\oplus_{\succ}E^{cs}$,  the subbundle $E^{cu}\oplus_{\succ}E^{cs}$ is  {\it mixed} if every Gibbs $u$-state has only negative Lyapunov exponents along $E^{cs}$ and  only positive Lyapunov exponents along $E^{cu}$. 
Partially hyperbolic diffeomorphisms characterized by   splittings of form
 $E^{uu}\oplus_{\succ}E^{cu}\oplus_{\succ}E^{cs}$ with mixed centers  are the central focus of the study in \cite{CM,ref7}. When $E^{cu}$ is uniformly expanding,  the results in \cite{CM} show that the number of physical measures also coincides with the cardinality of any strong-skeleton.

For any positive integer $n>1$, Dolgopyat, Viana and Yang choose an appropriate map $\hat{f}$ such that $E^c\oplus E^{ss}$ is mostly contracting and the cardinality of some strong-skeleton is $n$(with compatible choices). This choice of  $\hat{f}$ allows the system to have $n$  physical measures,  with  a one-to-one corresponding between fixed points in   the strong-skeleton of $\hat{f}$ and physical measures.  

 By  selecting a suitable transitive $C^{1+}$-Anosov map $\hat{T}$ such that
$$
TM=E^{uu}\oplus_{\succ}E^{u}\oplus_{\succ} (E^c\oplus E^s \oplus E^{ss}),
$$
it can be ensured that the subbundle $E^c\oplus E^s \oplus E^{ss}$ of $TM$ is mostly contracting in the product system $(\hat{f}\times\hat{T}, \T^2\times S^2\times\T^2)$.  Consequently, by Remark~\ref{mostlycontra},  $\hat{f}\times\hat{T}$ satisfies  general assumptions outlined in \cite{CM} by choosing  the subbundle $E^u$ of $TM$ to be $E^{cu}$.   In addition,  the cardinality of strong-skeleton of $\hat{f}\times\hat{T}$ is  equal to the cardinality of strong-skeleton of $\hat{f}$, as established in the proof of Proposition~\ref{pro.con} (see Remark~\ref{mostlycontra}).
Furthermore, based on the results of \cite{CM},    the number of  physical measures of $\hat{f}\times\hat{T}$ coincides with  the cardinality of strong-skeleton of system $(\hat{f}\times\hat{T}, M)$, which is $n$.

The properties of having a mostly contracting center and a mixed center are $C^1$-open among $C^{1+}$-diffeomorphisms.  Such properties actually guarantee that the number of physical measures coincides with the cardinality of any strong-skeleton in various partially hyperbolic splitting, including $E^{uu}\oplus_{\succ} E^{cs}$ and $E^{uu}\oplus_{\succ}E^{cu}\oplus_{\succ} E^{cs}$.

 Moreover, Dolgopyat, Viana and Yang explained how the physical measures of $\hat{f}$ collapse under small perturbations of the diffeomorphism.  
Specifically, for any $1\le m<n$,    a locally supported $C^{\infty}$-perturbation made on the neighborhood of the unstable manifold of a fixed point in the strong-skeleton of $\hat{f}$ transforms the original map $\hat{f}$  into $\hat{f}_\varepsilon$. This perturbation causes the cardinality of the strong-skeleton of $\hat{f}_\varepsilon$ to be $m$.
 Consequently, by Remark~\ref{mostlycontra}  (an alternative version of Proposition~\ref{pro.con}),  the number of physical measures associated with the perturbed system $\hat{f}_\varepsilon\times\hat{T}$ is also $m$. This provides a clear example of the upper semi-continuous variation in the number of physical measures among partially hyperbolic diffeomorphisms with   splittings of the form “$E^{uu}\oplus_{\succ}E^{cu}\oplus_{\succ}E^{cs}$” and mixed centers, as established in the theoretical results of \cite{CM}.  In addition, when \( E^{cu} \) is not uniformly expanding and \( E^{cs} \) is not uniformly contracting, we can also construct specific examples to demonstrate the variation of physical measures, in accordance with the premises and results of the 'mixed' hypothesis in reference \cite{CM}.

\section{The Connection Between  Linear Anosov Skew-Products and a $c$-Mixed Center}

 In this section,  we analyze the possibility of  constructing partially hyperbolic diffeomorphisms with  a $c$-mixed center from the perspective of  skew-product structures. The details are as follows.

Let $M$ be a compact smooth Riemannian manifold.
Let $f$ be a diffeomorphism on the  manifold $M$.  Let \( G_1 \) and \( G_2 \) be two continuous subbundles whose intersection is trivial, i.e., it contains only the zero vector. For simplicity, we will omit the base points of tangent vectors in what follows.
For any $\varepsilon>0$,  define a cone
 $$\C_{\varepsilon}(G_2, G_1)=\{v_1+v_2\in G_1\oplus G_2:\|v_1\|\le\varepsilon\|v_{2}\|\},$$
We call $\C_{\varepsilon}(G_2,G_1)$  {\it cone field of width $\alpha$ around $G_2$}. The cone $\C_{\varepsilon}(G_2, G_1)$ is {\it $Df$-forward invariant}
if there exists a constant $\theta\in(0,1)$ such that
\begin{itemize}
\item   $Df(\C_{\varepsilon}(G_2, G_1))\subset\C_{\theta\varepsilon}(G_2, G_1)$;
\item $Df(G_1\oplus G_2)=G_1\oplus G_2$.
\end{itemize}
It is well-known (see \cite{Yoccoz}) that if $\C_{\varepsilon}(G_2, G_1)$ is  $Df$-forward invariant,
then there exist subbundles $\hat{G_2}, \hat{G_1}$  such that $$\hat{G_2}\oplus_\succ \hat{G_1}, \hat{G_1}\subset G_1\oplus G_2, \hat{G_2}\subset G_1\oplus G_2 \quad \text{and} \quad \hat{G_2}\subset\C_{\varepsilon}(G_2, G_1)$$ Furthermore, $\C_{\varepsilon}(G_2, G_1)$ is {\it $Df$-unstable} if it is $Df$-forward invariant  and there exist $C$ and $\gamma>1$ such that for any $n$, $v\in \C_{\varepsilon}(G_2, G_1)$,
$$
\|Df^n(v)\|\ge C\gamma^n\|v\|,
$$
then $\hat{G_2}$ is uniformly expanding.

(The advantage of employing cones to determine whether a domination exists is that it suffices to verify that
\[
Df\bigl(\mathcal{C}_\varepsilon(G_2, G_1)\bigr) \subset \mathcal{C}_{\theta\varepsilon}(G_2, G_1).)
\]

Let $A:\T^2\to\T^2$ be a hyperbolic linear automorphism with the maximum eigenvalue $\lambda^u$. The automorphism $A$ admits a hyperbolic splitting $T\T^2=E^u_A\oplus E^s_A$.
In this section, we consider a $C^{1+}$-partially hyperbolic
skew-product diffeomorphism $f:\T^2\times S^1\to \T^2\times S^1$ defined by $f(x,y)=(A(x),K_x(y))$ with the following properties:
\begin{itemize}
\item The tangent bundle $TM$ admits a continuous splitting $TM=E^u_A\oplus E^c\oplus E^s_A$, where $E^c=TS^1$.

\item There exists a constant $\alpha>0$ such  that $\C_\alpha(E^u_A,E^c)$, $\C_\alpha(E^s_A, E^c)$ are $Df$-unstable, $Df^{-1}$-unstable, respectively.  (In fact, the $Df$-positive-invariance of $\C_\alpha(E^u_A,E^c)$ implies that $\C_\alpha(E^u_A,E^c)$ is $Df$-unstable, as shown in the proof of Lemma~\ref{ifonly}.)
\end{itemize}
A skew-product  with these properties above, where the action on the base space (or the first space in the product space) is a hyperbolic linear automorphism,  is called a {\it  linear Anosov skew-product}.  (This definition can be generalized to $\T^d\times N$, where $N$ is a compact smooth Riemannian manifold.)
It is well known that $f$ admits a partially hyperbolic splitting $E^u\oplus_{\succ} E^c\oplus_{\succ} E^s$ such that 
\begin{equation}\label{subbundle}
E^u\subset\C_\alpha(E^u_A,E^c), E^s\subset\C_\alpha(E^s_A, E^c).
\end{equation} 

\begin{lem}\label{ifonly}
For any invariant measure $\mu$, $\mu$ is a Gibbs $u$-state of $f$ if and only if $\mu$ is a $c$-Gibbs $u$-state of $f$.
\end{lem}
\begin{proof}
	For any $C^{1+}$-partially hyperbolic diffeomorphism with a partially hyperbolic splitting $E^{u}\oplus_{\succ}E^c\oplus_{\succ}E^{s}$, it has been established (see \cite{Yan}) that an invariant measure $\nu_1$ is a Gibbs $u$-state if and only if 
\begin{equation}\label{u-state}
h_{\nu_1}^u(f)=\int\log|\det(Df|_{E^{u}(x)})|d\nu_1.
\end{equation}

	For any invariant measure $\mu$, applying the Birkhoff ergodic theorem gives: 
	\begin{equation}\label{equ2}
	\int\log|\det(Df|_{E^{u}(x)})|d\mu=\int\lim_{n\to+\infty}\frac{\sum_{0\le i\le n-1}\log|\det Df|_{E^{u}(f^i(x))}|}{n}d\mu=\int\lim_{n\to+\infty}\frac{\log|\det Df^n|_{E^{u}(x)}|}{n}d\mu.
   \end{equation}
	
	Since $\C_\alpha(E^u_A,E^c)$ is $Df$-positive-invariant and $E^{u}\subset\C_\alpha(E^u_A,E^c)$ (by \ref{subbundle}),
	we have for  any $v\in\C_\alpha(E^u_A,E^c)$ (which can be expressed as  $v^u+v^c$ with  $v^{\sigma}\in E^\sigma$, $\sigma=u,c$): 
	\[\|A^n(v^u)\|\le\|Df^n(v)\|\le\sqrt{\alpha^2+1}\|A^n(v^u)\|,\]
	which implies
	\[\frac{1}{\sqrt{\alpha^2+1}}\cdot\frac{\|A^n(v^u)\|}{\|v^u\|}\le\frac{\|Df^n(v)\|}{\|v\|}\le\sqrt{\alpha^2+1}\cdot\frac{\|A^n(v^u)\|}{\|v^u\|}.\]
	Since $\dim(E^{u})=1$, it follows that
		\[\frac{1}{\sqrt{\alpha^2+1}}\cdot|\det(A^n|_{E^{u}_A})|\le|\det(Df^n|_{E^{u}})|\le\sqrt{\alpha^2+1}\cdot|\det(A^n|_{E^{u}_A})|.\]
	Consequently, we have:
	\begin{equation}\label{equ3}
		\lim_{n\to+\infty}\frac{\log|\det Df^n|_{E^{u}(x)}|}{n}=\log\lambda^u.
	\end{equation}
Thus,
\begin{equation}\label{inva}
\int\log|\det(Df|_{E^{u}(x)})|d\mu=\log\lambda^u.
\end{equation}

An invariant measure $\nu_2$ is a $c$-Gibbs $u$-state if and only if 
\begin{equation}\label{c-u-state}
h_{\nu_2}^u(f)=h_{top}(A)=\log\lambda^u.
\end{equation}

	By combining equations \eqref{u-state}, \eqref{c-u-state}, and \eqref{inva}, we conclude the proof of the lemma.
\end{proof}

\subsection{$c$-“Mostly” Expanding Subbundle}

Recall that $E^c$ is {\it mostly expanding} if every Gibbs $u$-state has only positive  Lyapunov exponents along $E^{c}$.

\begin{lem}\label{volum}
There exists a Gibbs $u$-state of $f$  such that the Lyapunov exponents along the  subbundle  $E^c$ are non-positive on a set of positive measure.
\end{lem}
\begin{proof}
Consider any $v\in\C_\alpha(E^u_A,E^c)$, which can be expressed as $v=v^u+v^c$  where $v^{\sigma}\in E^\sigma$ for $\sigma=u,c$. From this, by $Df$-invariance of $\C_\alpha(E^u_A,E^c)$ and $Df(E^c)=E^c$, we obtain:
	\begin{equation}\label{lya}
		\|A^n(v^u)\|\le\|Df^n(v)\|\le\sqrt{\alpha^2+1}\|A^n(v^u)\|.
	\end{equation}
Consequently, for 
	 any $v^{u}\in E^u$ and any $v^{s}\in E^s$, we can deduce:
	\[\lim_{n\to+\infty}\frac{\log\|Df^nv^{u}\|}{n}=\log\lambda^u, \lim_{n\to+\infty}\frac{\log\|Df^{-n}v^{s}\|}{n}=\log\lambda^u.\]
	
	Assuming the contrary, we can invoke   \cite[Theorem~C]{AV}, which states that there exists some physical measure, denoted by $\mu$. Notice that $\mu$ is also  an ergodic Gibbs $u$-state. By   \cite[Theorem~3.1]{SMMJ},  the physical measure $\mu$ must be volume non-expanding in the sense that
	\begin{equation}\label{contr1}
		\int \log|\det Df|d\mu\le 0.
	\end{equation}
	On the other hand, applying the Oseledets  multiplicative ergodic theorem(see \cite{Viana}), we obtain:
	\begin{equation}\label{contr2}
		\int \log|\det Df|d\mu=\log\lambda^u-\log\lambda^u+\log\lambda^c=\log\lambda^c>0,
	\end{equation}
	where $\log\lambda^c$ is the Lyanunov exponent along $E^c$ associated with the ergodic Gibbs $u$-state $\mu$ (by our hypothesis, $\log\lambda^c(\mu)>0$ for any ergodic Gibbs $u$-state $\mu$). 

It is obvious that inequality~\ref{contr2} contradicts inequality~\ref{contr1}. Therefore,  $E^c$ is not mostly expanding.
\end{proof}

\begin{lem}\label{mostlyexpand}
There exists a $c$-Gibbs $u$-state of $f$  such that the Lyapunov exponents along the  subbundle  $E^c$ are non-positive on a set of positive measure.
\end{lem}
\begin{proof}
By combining the results of Lemma~\ref{ifonly} and Lemma~\ref{volum}, we can directly derive this conclusion. \end{proof}

\subsection{$c$-“Mostly” Contracting Subbundle}

 In the above argument about “mostly” expanding subbundle, we have  shown that the subbundle $E^c$ is quite resistant to being $c$-mostly expanding.   Now, we turn our attention to the partially hyperbolic skew products (mainly linear Anosov skew-products), where we encounter  challenges in demonstrating that the corresponding subbundle is $c$-mostly contracting, particularly when selecting different  subbundles as the strong unstable bundle.

Consider a hyperbolic linear automorphism $B:\T^3\to\T^3$ characterized by the splitting of the tangent bundle $T\T^3= E^{uu}_B \oplus_{\succ} E^{u}_B \oplus_{\succ} E^{ss}_B$. The eigenvalues of $B$ are $\lambda^u_1, \lambda^u_2$, and $\lambda^s$ satisfying $\lambda^u_1 > \lambda^u_2 > 1 > \lambda^s > 0$.

By choosing $E^{uu}_B$ as the strong unstable subbundle in the definition of measure of maximal $u$-entopy, a measure $\mu$ is  of maximal $u$-entopy if and only if 
		\[h_\mu^u(B)=h^u_{top}(B)=\log\lambda^u_1=\int\log|\det(B|_{E^{uu}_B(x)})|d\mu.\]
This establishes that the measures of maximal $u$-entopy are equivalent to Gibbs $u$-states, where the space of Gibbs $u$-states is generated by choosing $E^{uu}_B$ as the strong unstable subbundle.  In this setting,  the system $(\T^3,B)$ may not factor over a hyperbolic linear  automorphism on a 2-dimensional torus.  

Consider a  $C^{1+}$-diffeomorphism $g:\T^3\times S^1\to \T^3\times S^1$ defined by $g(x,y)=(B(x),H_x(y))$, which is a linear Anosov skew-product.  The previous result in Lemma~\ref{ifonly} can be generalized to $g$, showing that  the measures of maximal $u$-entopy of $g$ are equivalent to Gibbs $u$-states of $g$. (We use Gibbs $u$-states as an intermediate equivalent due to the availability of numerous results related to Gibbs $u$-states in non-uniformly hyperbolic settings.)
Similar arguments regarding volume non-expansion (see the proof of Lemma~\ref{volum}) suggest that  $g$ can only be expected to exhibit a mostly  contracting  behavior  along $TS^1$, meaning every Gibbs $u$-state has only negative Lyapunov exponents along $TS^1$. 

Assuming that there exists an invariant torus $\T^3\times\{z\}$, we can  choose a proper $H_x$ to satisfy  the following inequality  for the Lebesgue measure on $\T^3 \times \{z\}$:
\begin{equation}\label{mu}
	\int _{\T^3\times\{z\}} \log\|DH_x(z)\|d\mu<0.
\end{equation} 

However, in  the simpler system $(\T^3,B)$, where the subbundle $E^{uu}_B$ corresponds to the strong unstable bundle in the general definition of measure of maximal $u$-entopy, 
we currently have no evidence to suggest that only the Lebesgue measure on $\T^3$ is the measure of maximal $u$-entopy (or Gibbs $u$-state).
Thus, verifying whether the  condition~\ref{mu} holds for all measures of maximal $u$-entopy in the system $(\T^3\times S^1,g)$ poses  additional challenges,  particularly regarding whether every Gibbs 
$u$-state exhibits only negative Lyapunov exponents along $TS^1$ (Note that the chosen strong unstable subbundle $E_g^{uu}$ is one-dimensional and satisfies $D\pi(E_g^{uu})=E^{uu}_B$, where $\pi$ is the projection map from $\T^3\times S^1$ to $\T^3$).
 At this point, we lack a feasible approach to prove corresponding Lemma~\ref{contained} and Lemma~\ref{lem.factoroveranosov}, mainly due to the difficulty in observing the strong-unstable disks or plaques. In comparison, the main reason behind the results in Proposition~\ref{pro.con} is that these strong-unstable disks (or plaques) can be observed within the  given product structure.

Therefore, we close this subsection by posing the  following intriguing questions:

\begin{qn}Let $B:\T^3\to\T^3$ be a hyperbolic linear  automorphism  with eigenvalues $\lambda^u_1,\lambda^u_2$ and $\lambda^s$ such that $\lambda^u_1>\lambda^u_2>1>\lambda^s>0$. 
Consider the   $C^{1+}$-linear Anosov skew-product $g:\T^3\times S^1\to \T^3\times S^1$ defined by $g(x,y)=(B(x),H_x(y))$. It is clear that $g$ admits a partially hyperbolic splitting 
$$
TM=E^{uu}_g\oplus_{\succ}E^u_g \oplus_{\succ}TS^1 \oplus_{\succ} E^{ss}_g.
$$  Can we choose a suitable family of maps $H_x:S^1\to S^1$ for each $x\in\T^3$ such that $TS^1$ is mostly contracting when  $E^{uu}_g$ is chosen as the strong unstable subbundle?  
\end{qn}

\subsection{The origin of Proposition~\ref{pro.con}}

Consider a linear Anosov skew-product $h:\T^2\times S^1_1\times P\to\T^2\times S^1_1\times P$, defined by $$h(x,y,z)=(A(x),H_x(y),K_{x,y}(z))),$$ where $P=S^1_2\times S^1_3$ and $S^1_1=S^1_2=S^1_3=S^1$ is the circle.  
The product space $S^1_1\times P$ is treated as $N$ in the definition of linear Anosov skew-products.  Let $h_1(x,y)=(A(x),H_x(y))$. Notice that $h_1$ is also a linear Anosov skew-product. Analogous to Lemma~\ref{ifonly}, the Gibbs \( u \)-states are equivalent to the \( c \)-Gibbs \( u \)-states in both dynamical systems \( (\mathbb{T}^1 \times S^1, h_1) \) and \( (\mathbb{T}^2 \times S^1_1 \times P, h) \).

When $K_{x,y}(z)$ is independent of $y$, notice that
\[
D h(x, y, z) =
\begin{pmatrix}
\frac{\partial A(x)}{\partial x} & 0 & 0 \\[6pt]
\frac{\partial H_x(y)}{\partial x} & \frac{\partial H_x(y)}{\partial y} & 0 \\[6pt]
\frac{\partial K_{x,y}(z)}{\partial x} & 0 & \frac{\partial K_{x,y}(z)}{\partial z}  
\end{pmatrix}.
\]
and $TS^1_1$ is still $Dh$-invariant. Then we have following lemma, which is therefore helpful in constructing such examples and  guides us to consider Proposition~\ref{pro.con}.

\begin{lem}\label{uniformexpanding1}
Suppose that:
\begin{itemize}
\item there exists a partially hyperbolic splitting 
$TM=E^{uu}\oplus_{\succ}E^{cu}_P\oplus_{\succ}(E^{cs}_P\oplus TS^1_1),$
where $M=\T^2\times S^1_1\times P$, $TP=E^{cu}_P\oplus_{\succ}E^{cs}_P$ and $E^{cu}_P$ is non-trivial;
\item $K_{x,y}(z)$ is independent of $y$, meaning that $K_{x,y}(z)=K_{x}(z)$ for any $y\in S^1_1$.
\end{itemize}
 Then $E^{cu}_P$ is uniformly expanding and $\dim(E^{cs}_P)=1$.
\end{lem}
\begin{proof}
Since $TM=E^{uu}\oplus_{\succ}E^{cu}_P\oplus_{\succ}(E^{cs}_P\oplus TS^1_1)$,  assume that there exists a constant $\lambda\in(0,1)$ such that for any point $(x,y,z)$,
\begin{equation}\label{subdomination}
 \frac{\|Dh|_{TS^1_1(x,y,z)}\|}{\|Dh|_{E^{cu}_P(x,y,z)}\|}=\frac{|DH_x(y)|}{\|Dh|_{E^{cu}_P(x,y,z)}\|}\le\lambda.
\end{equation}
Notice that for any $x\in\T^2$, $H_x:S^1_1\to S^1_1$ is a diffeomorphism. There always exists a point $y_x$ such that $|DH_x(y_x)|\ge1$. Combining  the independence  of $\|Dh|_{E^{cu}_P(x,y,z)}\|$ on $y$ and inequality~\ref{subdomination},  it follows that 
$$
\|Dh|_{E^{cu}_P(x,y,z)}\|\ge\frac{|DH_x(y_x)|}{\lambda}\ge\frac{1}{\lambda}.
$$

Assume that \( E^{cs}_P = 0 \). Then, \( E^{cu}_P \) is uniformly expanding on \( TP \), which would contradict the fact that each \( K_x : P \to P \) is a diffeomorphism.
\end{proof}

This lemma also tells us that for \( E^{cu} \) to exhibit non-uniform expansion, \( K_{x,y}(z) \) must depend smoothly (at least \( C^1 \)) on \( (x,y) \). In this case, \( TS^1 \) will no longer necessarily be invariant under \( Dh \). Based on this, we have the following lemma.

\begin{lem}\label{equvilence}
Suppose that there exists a partially hyperbolic splitting 
$TM=E^{uu}\oplus_{\succ}E^{cu}_P\oplus_{\succ}T\tilde{S^1_1}\oplus_{\succ}E^{cs}_P,$
where $M=\T^2\times S^1_1\times P$, such that
\begin{itemize}
\item $TP=E^{cu}_P\oplus_{\succ}E^{cs}_P$;
\item $Df(TS^1\oplus E^{cu}_P)=TS^1\oplus E^{cu}_P$, $T\tilde{S^1_1}\subset TS^1\oplus E^{cu}_P$.
\item Any $c$-Gibbs $u$-state has only positive the Lyapunov exponents along $E^{cu}_P$ in system \((\mathbb{T}^2 \times S^1_1 \times P, h)\). 
\item The Lyapunov exponents along $TS^1_1$ for each c-Gibbs u-state are negative in the system $(\T^2 \times S^1_1, h_1)$.
\end{itemize}
Then,  the Lyapunov exponents along $T\tilde{S^1_1}$ for each $c$-Gibbs $u$-state are  negative in the system $(\T^2 \times S^1_1 \times P, h)$.
\end{lem}
\begin{proof}
Notice that
\[
D h(x, y, z) =
\begin{pmatrix}
\frac{\partial A(x)}{\partial x} & 0 & 0 \\[6pt]
\frac{\partial H_x(y)}{\partial x} & \frac{\partial H_x(y)}{\partial y} & 0 \\[6pt]
\frac{\partial K_{x,y}(z)}{\partial x} & \frac{\partial K_{x,y}(z)}{\partial y} & \frac{\partial K_{x,y}(z)}{\partial z}  
\end{pmatrix},
\]
and the inverse of the $Dh$ remains in {\color{red}Lower} triangular form.
Then for any $z\in P$, since $Df(TS^1\oplus E^{cu})=TS^1\oplus E^{cu}$, $E^{cu}\oplus_{\succ} TS^1$, we can check that
\begin{equation}\label{dayudengyu}
\liminf_{n\rightarrow+\infty}\frac{1}{n}\log\| Dh_1^{-n}|_{TS^1_1(x,y)}\|\le \liminf_{n\rightarrow+\infty}\frac{1}{n}\log\| Dh^{-n}|_{TS^1_1(x,y,z)}\|.
\end{equation}

Let $\mu$ be an ergodic $c$-Gibbs $u$-state in the system $(\T^2\times S^1_1\times P,h)$.
Let $\pi_{12}:\T^2\times S^1_1\times P\to\T^2\times S^1_1$ be the projection defined by  $\pi_{12}(x,y,z)=(x,y)$. Leb $\F^{u}(h)$ be the strong unstable foliation of $h$. By the smoothness of \( \pi_{12} \), it follows from Equation~\ref{checku} in Subsection~\ref{checkuu} that \( \pi_{12}(\mathcal{F}^{u}(h)) \) is  tangent to the strong unstable subbundle of \( h_1 \) and \( h_1 \)-invariant (by $\pi_{12}\circ h=h_1\circ\pi_{12}$). By the uniqueness of the strong unstable foliation of $h_1$, it follows that the projection \(\pi_{12}\) (smoothly) maps each leaf of the strong unstable foliation in the system \((\mathbb{T}^2 \times S^1_1 \times P, h)\) onto a corresponding leaf of the strong unstable foliation in the system \((\mathbb{T}^2 \times S^1_1, h_1)\).
  Then the proofs analogous to those of Lemma~\ref{contained} also show that $\pi_{12*}\mu=\nu$, where $\nu$ is an ergodic $c$-Gibbs $u$-state in the system $(\T^2 \times S^1_1, h_1)$.
It follows that there exist a constant $a>0$ and $\nu$-full measure set $\Lambda$ such that any $x\in\Lambda$, 
$$
\liminf_{n\rightarrow+\infty}\frac{1}{n}\log\| Dh_1^{-n}|_{TS^1_1(x,y)}\|>a>0.
$$
Since $\pi_{12*}\mu=\nu$, it follows that $\mu(\pi^{-1}(\Lambda))=1$ and any point $(x,y,z)\in \pi^{-1}(\Lambda)=\Lambda\times P$, by relation~\ref{dayudengyu} we have
\[
\liminf_{n\rightarrow+\infty}\frac{1}{n}\log\| Dh^{-n}|_{TS^1_1(x,y,z)}\|\ge\liminf_{n\rightarrow+\infty}\frac{1}{n}\log\| Dh_1^{-n}|_{TS^1_1(x,y)}\|>a.
\]

By our assumption $T\tilde{S^1_1}\subset TS^1\oplus E^{cu}_P$ and $T\tilde{S^1_1}\neq E^{cu}_P$, for any $v\in T\tilde{S^1_1}$, we can write $v=v_1+v_2$, where $v_1\in TS^1$ and $v_2\in E^{cu}_P$ ($v_1\neq0$). For the ergodic $c$-Gibbs $u$-state $\mu$, we can see that there exists a $\mu$-full measure set $\Lambda(\mu)\subset \Lambda\times P$ such that any point $c\in \Lambda(\mu)$, we have
\[
\lim_{n\rightarrow+\infty}\frac{1}{n}\log\| Dh^{-n}(v_1(c))\|>a>0 \quad \text{and} \quad \lim_{n\rightarrow+\infty}\frac{1}{n}\log\| Dh^{-n}(v_2(c))\|<0. 
\]
However, $v(c)=v_1(c)+v_2(c)$, and  the lyapunov exponents of $h^{-1}$ along $v_1(c)$ and $v_2(c)$ is different, by the corresponding theory of lyapunov exponents,  it follows that 
\[
\lim_{n\rightarrow+\infty}\frac{1}{n}\log\| Dh^{-n}(v(c))\|=\max\{\lim_{n\rightarrow+\infty}\frac{1}{n}\log\| Dh^{-n}(v_1(c))\|, \lim_{n\rightarrow+\infty}\frac{1}{n}\log\| Dh^{-n}(v_2(c))\|\}>0.
\] 

Combining this result with item~\ref{d} of Proposition~\ref{main.pro}, we can establish this implication. 
\end{proof}

Lemma~\ref{equvilence} also helps us understand the form of Proposition~\ref{pro.con}, as it provides a simpler construction within Lemma~\ref{equvilence} (particularly when considering the linear Anosov skew-product). More specifically, Lemma~\ref{equvilence} can be used to construct examples of \( E^{cu} \) that is not uniformly expanding, corresponding to a skew-product system on a skew-product system. In the application presented in the next section, we will construct such an example.

\section{Construction of $c$-“Mostly”mixed center and Proof of Theorem~C}

\subsection{Construction}

Observe that if \( f \) admits a partially hyperbolic splitting \( TM = E^{uu} \oplus_\succ E^{cu} \oplus_\succ E^{cs} \) (for simplicity, we first assume \( \dim(E^{cs}) = \dim(E^{cu}) = 1 \), which differs from the dimensions in the construction below), then by the invariance of \( E^{cu} \) and \( E^{cs} \), for any fixed point \( x \), there exists a constant \( \lambda \in (0,1) \) such that  
\[
\frac{\|Df|_{E^{cs}(x)}\|}{m(Df|_{E^{cu}(x)})} \leq \lambda.
\]  
Here, \( m(Df|_{E^{cu}(x)}) \) denotes the minimal norm of \( Df|_{E^{cu}(x)} \).  
Further inspired by Smale's work \cite{Smale}, we present the following construction.

Fix any arbitrarily small  $\delta>0$ satisfying that 
\[ \Leb([-2\delta,2\delta]^2)\le\frac{1}{100}, \]
where \( \Leb(\cdot) \) denotes the Lebesgue measure in the two-dimensional plane,
 there exists a $C^{\infty}$-smooth truncation function $s$  on $\R$ such that
\begin{itemize}
    \item When \( x \in (-\infty, +\infty) \), \( s(x) = s(-x) \) (symmetric about \( x = 0 \)).
    \item  Map \( s(x) \) is strictly monotonic on $\left(\frac{\delta}{2}, \delta\right)$.
    \item When \( x \in [0, \frac{\delta}{2}] \), \( s(x) = 1 \), and when \( x \in [\delta, +\infty) \), \( s(x) = 0 \).
\end{itemize}
{\color{red} Some of the detailed constructions related to $s$ have been omitted.
}

After fixing \( \delta, s(\cdot) \), the function
\[
s(x)\cdot (s(y) + y \cdot s'(y)):\R\times\R\to\R
\]
is bounded below. There exists $M>0$ such that
\[
-M \leq s(x)\cdot (s(y) + y \cdot s'(y)).
\]

Let \( A = D^n \times D^m \) act on \( \mathbb{T}^2 \times \mathbb{T}^2 \) for sufficiently large \( n \) and \( m \), where  
\[  
D = \begin{pmatrix} 2 & 1 \\ 1 & 1 \end{pmatrix}, \quad \text{with} \quad n > m,  
\]  
such that \( A \) satisfies the following assumptions:
\begin{itemize}
    \item \( A \) has eigenvalues \( 0<\lambda^{ss} < \lambda^s < 1 < \lambda^u < \lambda^{uu} \), such that
   \begin{equation}\label{-M}
    -M(1 - \lambda^u) + \lambda^u \le \frac{\lambda^{uu}}{2} \quad \text{and} \quad -M(1 - \frac{1}{\lambda^s}) + \frac{1}{\lambda^s} \le \frac{1}{2\lambda^{ss}}.
    \end{equation}
    \item The eigenvalues \( \lambda^{ss}, \lambda^s, \lambda^u, \lambda^{uu} \) correspond to the eigenspaces \( E^{ss}, E^{s}, E^{u}, E^{uu} \), respectively. The foliations that are tangent to these eigenspaces at each point are denoted by \( \mathcal{F}^{ss}, \mathcal{F}^s, \mathcal{F}^u, \mathcal{F}^{uu} \), respectively.
    \item \( A \) has two fixed points, \( p \) and \( q \) such that there exist  disjoint open neighborhoods $U_p$ and $U_q$ of points $p$ and $q$, respectively, satisfies that  $\Lambda(p) \subset U_p$ and $\Lambda(q)\subset U_q$, where 
\[
\Lambda(p)=\F^{uu}_{2\delta}(p)\times \F^{ss}_{2\delta}(p)\times \F^u_{2\delta}(p)\times \F^{s}_{2\delta}(p) \quad \text{and} \quad \Lambda(q)=\F^{uu}_{2\delta}(q)\times \F^{ss}_{2\delta}(q)\times \F^u_{2\delta}(q)\times \F^{s}_{2\delta}(q).
\]
 \item There exists a small constant \( 0<\varepsilon_1<1 \) such that  
\begin{equation}\label{condition}
[\log\frac{\sqrt{1-\varepsilon_1}}{\sqrt{\varepsilon_1^2+1}}] \cdot \frac{1}{10} + [\log\left(\frac{\lambda^u}{\sqrt{\varepsilon_1^2+1}}\right)] \cdot \frac{9}{10} > 0  
\quad \text{and} \quad  
\frac{\sqrt{1-\varepsilon_1}}{\sqrt{\varepsilon_1^2+1}}<1<\frac{\lambda^u}{\sqrt{\varepsilon_1^2+1}}.
\end{equation}
\end{itemize}

Notice that the eigenspaces of  $A$ are orthogonal to each other. By diagonalizing the coordinates (locally), we may assume that for each fixed point $x\in\{p,q\}$
$$
A(a,b,c,d)=(\lambda^{uu}a, \lambda^{ss}b, \lambda^uc,\lambda^sc): \F^{uu}(x)\times \F^{ss}(x)\times \F^u(x)\times \F^{s}(x)\to \F^{uu}(x)\times \F^{ss}(x)\times \F^u(x)\times \F^{s}(x),
$$
and 
$$
A^{-1}(a,b,c,d)=(\frac{a}{\lambda^{uu}}, \frac{b}{\lambda^{ss}}, \frac{c}{\lambda^u},\frac{d}{\lambda^s}): \F^{uu}(x)\times \F^{ss}(x)\times \F^u(x)\times \F^{s}(x)\to \F^{uu}(x)\times \F^{ss}(x)\times \F^u(x)\times \F^{s}(x).
$$

 Next, we define the following mapping \( I_\varepsilon \) on $\T^4$ (For simplicity, we can assume that \( \Lambda(p) = [-2\delta, 2\delta]^4 \) when considering the map in \( U_p \) and $p=(0,0,0,0)$. Similarly, we can also assume that \( \Lambda(q) = [-2\delta, 2\delta]^4 \) when considering the map in \( U_q \)) and $q=(0,0,0,0)$: 
\begin{itemize}
\item when $(a,b,c,d)\notin U_p\cup U_q$, $I_\varepsilon=I_\varepsilon^{-1}=I$, where $I$ is the idenity;
\item when $(a,b,c,d)\in U_p$, $I_\varepsilon=(a, b, \frac{P(a,b,c,d)}{\lambda^{u}},d)$ on $\Lambda(p)$, where $$P(a,b,c,d)=s(kc)\cdot s(\sqrt{a^2+b^2+d^2})\cdot(c-\lambda^uc)+\lambda^uc.$$
\item when $(a,b,c,d)\in U_q$, $I_\varepsilon^{-1}=(a, b, c,\lambda^s\cdot Q(a,b,c,d))$ on $\Lambda(q)$,  where 
$$
Q(a,b,c,d)=s(kd)\cdot s(\sqrt{a^2+b^2+c^2})\cdot(d-\frac{1}{\lambda^s}d)+\frac{1}{\lambda^s}d.
$$
\end{itemize}
We begin by proving the following lemma. Then, we demonstrate that \( I_\varepsilon \) is well-defined and $C^\infty$.

\begin{lem}\label{splitting}
We have the following inequalities:
$$
1\le\frac{\partial P}{\partial c}\le\frac{\lambda^{uu}}{2} \quad \text{and} \quad 1\le\frac{\partial Q}{\partial d}\le \frac{1}{2\lambda^{ss}}.
$$
\end{lem}
\begin{proof}
Since
\[
{\color{red}  \frac{\partial P}{\partial c}} = s(\sqrt{a^2+b^2+d^2}) \cdot (1 - \lambda_2) \cdot \left[ s(kc) + k c \cdot s'(kc) \right] + \lambda_2.
\]
Due to our construction, it follows that 
$$
(1-\lambda_2)\cdot k c \cdot s'(kc)\ge0 \quad \text{and} \quad 0\le s(\sqrt{a^2+b^2+d^2})\cdot s(kc)\le1.
$$
Then
$$
{\color{red} \frac{\partial P}{\partial c}}\ge  s(\sqrt{a^2+b^2+d^2}) \cdot (1 - \lambda_2) \cdot s(kc)+\lambda_2\ge1.
$$
 Notice that the quantity
\[
s(\sqrt{a^2+b^2+d^2})\cdot (s(kc) + k c \cdot s'(kc))
\]
is bounded below after fixing \( s \). By our assumption~\ref{-M}, there exists $M>0$ such that
\[
-M \leq s\left(\sqrt{a^2 + b^2+d^2}\right) \cdot \left(s(kc) + k c \cdot s'(kc)\right).
\]
Due to the properties of the hyperbolic linear automorphism we have chosen, it follows that
$${\color{red} \frac{\partial P}{\partial c}}\le -M(1-\lambda^u)+\lambda^u\le\frac{\lambda^{uu}}{2}.$$

Similarly, we can verify that
\[
1\le{\color{red} \frac{\partial Q}{\partial d}}\le \frac{1}{2\lambda^{ss}}.
\]
\end{proof}

\begin{lem}\label{well-defined}
$I_\varepsilon(U_p)=U_p$, $I_\varepsilon$ is $C^\infty$ on $U_p$ and a bijection on $U_p$; $I_\varepsilon^{-1}(U_q)=U_q$, $I_\varepsilon^{-1}$ is $C^\infty$ on $U_q$  and a bijection on  $U_q$. Consequently, $I_\varepsilon$ and $I_\varepsilon^{-1}$ are $C^\infty$ on $\T^4$ and $I_\varepsilon\circ I_\varepsilon^{-1}=I$. 
\end{lem}
\begin{proof}
When \( x \notin U_p \), \( I_\varepsilon^{-1} \) can be determined. We only need to show that \( I_\varepsilon \) is bijective on \( U_p \) since $I_\varepsilon(x)=I(x)$ when $x\in U_p\setminus\Lambda(p)$. If this condition is satisfied, then \( I_\varepsilon \) is well-defined (similarly for $I_\varepsilon^{-1}$). To demonstrate this, it suffices to show that \( I_\varepsilon \) is bijective  on the  cube $\Lambda(p)=[-2\delta,2\delta]^4$.

By Lemma~\ref{splitting}, we have \( \frac{\partial P}{\partial c} \ge 1 \). Thus, for any given values of \( a \), \( b \), and \( d \), the function \( \frac{P}{\lambda^u} \) is strictly monotonic about \( c \) on the interval \( [-2\delta, 2\delta] \), mapping this interval onto itself. Consequently, it is easy to check that $I_\varepsilon$ is a bijection on \( U_p \) (a similar argument holds for \( I_\varepsilon^{-1} \) on \( U_q \)).

Since \( s \) is \( C^\infty \) (infinitely differentiable), the map \( I_\varepsilon \) on \( U_p \) and \( I_\varepsilon^{-1} \) on \( U_q \) are both \( C^\infty \)-smooth. By \cite[Proposition~5.7]{Lee}, both \( I_\varepsilon \) on \( U_p \) and \( I_\varepsilon^{-1} \) on \( U_q \) admit smooth inverses. Since \( I_\varepsilon = I_\varepsilon^{-1} = I \) when \( x \notin U_p \cup U_q \) and \( I_\varepsilon = I_\varepsilon^{-1} = I \) when \( x \in (U_p \setminus\Lambda(p))\cup (U_q  \setminus \Lambda(q)) \), it follows that \( I_\varepsilon \) and \( I_\varepsilon^{-1} \) are mutually compatible and can be defined on the entire space \( \mathbb{T}^4 \), with both being \( C^\infty \)-smooth.
\end{proof}

Now, Let us construct a mixed partially hyperbolic diffeomorphism $f$.
More presicely, for sufficiently large $k$  (which we will choose appropriately later),  we define $f=A\circ I_\varepsilon$.  Based on the definition of \( I_\varepsilon \), we obtain the following:
{\color{red}
\[
Df|_{\Lambda(p)} = \begin{pmatrix}
\lambda^{uu} & 0 & 0 & 0 \\[6pt]
0 & \lambda^{ss} & 0 & 0 \\[6pt]
\frac{\partial P}{\partial a} & \frac{\partial P}{\partial b} & \frac{\partial P}{\partial c} & \frac{\partial P}{\partial d} \\[6pt]
0 & 0 & 0 & \lambda^s
\end{pmatrix} \quad \text{and} \quad
Df|_{\Lambda(q)} = \begin{pmatrix}
\lambda^{uu} & 0 & 0 & 0 \\[6pt]
0 & \lambda^{ss} & 0 & 0 \\[6pt]
0 & 0 & \lambda^u & 0 \\[6pt]
-\frac{\lambda^s\frac{\partial Q}{\partial a}}{\frac{\partial Q}{\partial d}} & -\frac{\lambda^s\frac{\partial Q}{\partial b}}{\frac{\partial Q}{\partial d}} & -\frac{\lambda^s\frac{\partial Q}{\partial c}}{\frac{\partial Q}{\partial d}} & (\frac{\partial Q}{\partial d})^{-1}
\end{pmatrix}.
\]
}

\smallskip\smallskip\smallskip\smallskip\smallskip\smallskip

Next, we sequentially prove that the map $f$ constructed in this way has the following properties:
\begin{enumerate}
\item\label{(1)}  $f$  is a $C^\infty$-diffeomorphism.
\item\label{(3)} $f$ is  partially hyperbolic with the partially hyperbolic splitting $$T\T^4=F^{uu} \oplus_{\succ} F^{cu} \oplus_{\succ} (F^{cs}\oplus_\succ F^{ss}).$$
\item\label{(4)} there exists fixed points $p,q$ such that 
\[
\|Df|{F^{cu}(p)}\| = 1 \quad \text{and} \quad \|Df|{F^{cs}(q)}\| = 1.
\]
\item\label{(5)} every Gibbs $u$-state exhibits  positive Lyapunov exponents along $F^{cu}$ and negative Lyapunov exponents along $F^{cs}\oplus F^{ss}$.
\end{enumerate}

Now, we are ready to prove above properties~\ref{(1)}\ref{(3)}\ref{(4)}\ref{(5)}.

\begin{proof}[Proof of item~\ref{(1)}]
Notice that $f$ is actually the composition of two maps, i.e., $f=A\circ I_\varepsilon$. Thus, by Lemma~\ref{well-defined}, we conclude the result.
\end{proof}

Now, we are ready to proof the item~\ref{(3)}.

\begin{proof}[Proof of item~\ref{(3)}]
Since
$$
P(a,b,c,d)=s(kc)\cdot s(\sqrt{a^2+b^2+d^2})\cdot(c-\lambda^uc)+\lambda^uc,
$$ and 
$$
Q(a,b,c,d)=s(kd)\cdot s(\sqrt{a^2+b^2+c^2})\cdot(d-\frac{1}{\lambda^s}d)+\frac{1}{\lambda^s}d.
$$
When the partial derivatives of \( P(a, b, c, d) \) with respect to \( a \), \( b \), and \( d \) and \( Q(a, b, c, d) \) with respect to \( a \), \( b \), and \( c \) are not zero,
\[
\frac{\partial P}{\partial a} = s(kc)\cdot c(1 - \lambda^u) \cdot s'( \sqrt{a^2 + b^2 + d^2} ) \cdot \frac{a}{\sqrt{a^2 + b^2 + d^2}},
\]
\[
\frac{\partial P}{\partial b} = s(kc)\cdot c(1 - \lambda^u) \cdot s'( \sqrt{a^2 + b^2 + d^2} ) \cdot \frac{b}{\sqrt{a^2 + b^2 + d^2}},
\]
\[
\frac{\partial P}{\partial d} = s(kc)\cdot c(1 - \lambda^u) \cdot s'( \sqrt{a^2 + b^2 + d^2} ) \cdot \frac{d}{\sqrt{a^2 + b^2 + d^2}}.
\]
\[
\frac{\partial Q}{\partial a} = s(kd) \cdot d( 1 - \frac{1}{\lambda^s})\cdot s'(\sqrt{a^2 + b^2 + c^2}) \cdot \frac{a}{\sqrt{a^2 + b^2 + c^2}}.
\]
\[
\frac{\partial Q}{\partial b} = s(kd) \cdot d( 1 - \frac{1}{\lambda^s})\cdot s'(\sqrt{a^2 + b^2 + c^2}) \cdot \frac{b}{\sqrt{a^2 + b^2 + c^2}}.
\]
\[
\frac{\partial Q}{\partial c} = s(kd) \cdot d( 1 - \frac{1}{\lambda^s})\cdot s'(\sqrt{a^2 + b^2 + c^2}) \cdot \frac{c}{\sqrt{a^2 + b^2 + c^2}}.
\]
Since
\[
 |s(kc)\cdot  c|\le\frac{\delta}{k} \quad \text{and} \quad |s(kd)\cdot d|\le\frac{\delta}{k}
\]
 and the function $s'(y):\R\to\R$ is bounded above, by choosing  $k$ sufficiently large,   the partial derivatives
\[
\frac{\partial P}{\partial a}, \quad \frac{\partial P}{\partial b}, \quad \frac{\partial P}{\partial d}, \quad \frac{\partial Q}{\partial a}, \quad \frac{\partial Q}{\partial b}, \quad \text{and} \quad \frac{\partial Q}{\partial c}
\]
can simultaneously be made arbitrarily close to \(0\).

By Lemma~\ref{splitting}, $$\lambda^s<1\le\frac{\partial P}{\partial c}\le\frac{\lambda^{uu}}{2}<\lambda^{uu} \quad \text{and} \quad \frac{1}{\lambda^u}<1\le\frac{\partial Q}{\partial d}\le\frac{1}{2\lambda^{ss}}<\frac{1}{\lambda^{ss}}.$$
Then, we can check that
\[
Df|_{\Lambda(p)} = \begin{pmatrix}
\lambda^{uu} & 0 & 0 & 0 \\[6pt]
0 & \lambda^{ss} & 0 & 0 \\[6pt]
\frac{\partial P}{\partial a} & \frac{\partial P}{\partial b} & \frac{\partial P}{\partial c} & \frac{\partial P}{\partial d} \\[6pt]
0 & 0 & 0 & \lambda^s
\end{pmatrix};
{\color{red}Df^{-1}|_{f(\Lambda(p))}}=\begin{pmatrix}
\frac{1}{\lambda^{uu}} & 0 & 0 & 0 \\[6pt]
0 & \frac{1}{\lambda^{ss}} & 0 & 0 \\[6pt]
-\left(\frac{\partial P}{\partial c}\right)^{-1} \frac{1}{\lambda^{uu}} \frac{\partial P}{\partial a} & -\left(\frac{\partial P}{\partial c}\right)^{-1} \frac{1}{\lambda^{ss}} \frac{\partial P}{\partial b}  &\left(\frac{\partial P}{\partial c}\right)^{-1} & -\left(\frac{\partial P}{\partial c}\right)^{-1}\frac{1}{\lambda^u} \frac{\partial P}{\partial d} \\[6pt]
0 & 0 & 0 & \frac{1}{\lambda^s}
\end{pmatrix}.
\]
and {\color{red} suppose that $R_a=-\frac{\lambda^s\frac{\partial Q}{\partial a}}{\frac{\partial Q}{\partial d}}, R_b=-\frac{\lambda^s\frac{\partial Q}{\partial b}}{\frac{\partial Q}{\partial d}}, R_c=-\frac{\lambda^s\frac{\partial Q}{\partial c}}{\frac{\partial Q}{\partial d}}$
\[
Df|_{\Lambda(q)} = \begin{pmatrix}
\lambda^{uu} & 0 & 0 & 0 \\[6pt]
0 & \lambda^{ss} & 0 & 0 \\[6pt]
0 & 0 & \lambda^u & 0 \\[6pt]
R_a & R_b & R_c & (\frac{\partial Q}{\partial d})^{-1}
\end{pmatrix};
Df^{-1}\big|_{f(\Lambda(q))} =
\begin{pmatrix}
\dfrac{1}{\lambda^{uu}} & 0 & 0 & 0 \\[8pt]
0 & \dfrac{1}{\lambda^{ss}} & 0 & 0 \\[8pt]
0 & 0 & \dfrac{1}{\lambda^u} & 0 \\[8pt]
-\dfrac{\partial Q}{\partial d}\,\dfrac{R_a}{\lambda^{uu}} & 
-\dfrac{\partial Q}{\partial d}\,\dfrac{R_b}{\lambda^{ss}} & 
-\dfrac{\partial Q}{\partial d}\,\dfrac{R_c}{\lambda^{u}} & 
\dfrac{\partial Q}{\partial d}
\end{pmatrix}.
\]
}

Therefore, we can choose a constant \( \varepsilon_0 > 0 \) and a sufficiently large integer \( k_0 \) such that for all \( k \geq k_0 \), the following holds: (Here, {\color{red}\( k \) depends on \( \varepsilon_0 \)}, and for any arbitrarily small \( \varepsilon_0 > 0 \), there always exists a sufficiently large \( k \) such that)
\begin{itemize}
\item $\C_{\varepsilon_0}(E^{uu}, E^u\oplus E^s\oplus E^{ss})$ is $Df$-unstable.
\item $\C_{\varepsilon_0}(E^{ss}, E^{uu}\oplus E^u\oplus E^{s})$ is $Df^{-1}$-unstable.
\item $E^u\oplus E^s$ is $Df$-invariant.
\item  $\C_{\varepsilon_0}(E^{u}, E^{s})$ is $Df$-forward invariant.
\item  $\C_{\varepsilon_0}(E^{s}, E^{u})$ is $Df^{-1}$-forward invariant.
\end{itemize}
Once the dimension of the dominated subbundles is fixed, the dominated splitting is uniquely determined.
It follows that there exists a partially hyperbolic splitting 
$T\T^4=F^{uu} \oplus_{\succ} F^{cu} \oplus_{\succ} (F^{cs}\oplus F^{ss})$ such that
\begin{itemize}
\item $F^{uu}\subset\C_{\varepsilon_0}(E^{uu}, E^u\oplus E^s\oplus E^{ss})$ and $F^{uu}$ is uniformly expanding.
\item $F^{ss}\subset \C_{\varepsilon_0}(E^{ss}, E^u\oplus E^s\oplus E^{uu})$ and $F^{ss}$ is uniformly contracting.
\item $F^{cu}\subset\C_{\varepsilon_0}(E^{u}, E^{s}), F^{cs}\subset\C_{\varepsilon_0}(E^{s}, E^{u})$.
\end{itemize}
\end{proof}

Now,  we explain the item~\ref{(4)}.

\begin{proof}[Proof of item~\ref{(4)}]
Since $$
Df_p=\begin{pmatrix}
\lambda^{uu} & 0 & 0&0 \\
0  & \lambda^{ss}& 0 &0\\
0 & 0 & 1&0\\
0&0&0&\lambda^s
\end{pmatrix}
\quad \text{and} \quad
Df_q=\begin{pmatrix}
\lambda^{uu} & 0 & 0&0 \\
0  & \lambda^{ss}& 0 &0\\
0 & 0 & \lambda^u&0\\
0&0&0&1
\end{pmatrix}
$$
 and $p,q$ are fixed points.
By domination, it follows that $F^{cu}(p)=E^{u}(p)$ and $F^{cs}(p)=E^{s}(p)$ (similarly for $q$). Therefore, the result is obtained directly.
\end{proof}

Before we prove the final item~\ref{(5)}, let us first prove the following lemma.
Note that for the modified map \( f \), both \( E^u \) and \( E^s \) are no longer invariant subbundles. However, we have \( F^{cu} \subset \mathcal{C}_{\varepsilon_0}(E^u, E^s) \) and \( F^{cs} \subset \mathcal{C}_{\varepsilon_0}(E^s, E^u) \). Therefore, we proceed with the following proof to determine the signs of the Lyapunov exponents of the Gibbs \( u \)-states along \( F^{cu} \) and \( F^{cs} \), respectively.

\begin{lem}\label{zuihouzhuyao} There exists sufficiently large $k$ such that
for any Gibbs $u$-state $\mu$ of $f$, we have that $$\int\log|\det Df|_{F^{cu}}|d\mu>0, \int\log|\det Df^{-1}|_{F^{cs}}|d\mu>0.$$
\end{lem}
\begin{proof}
 Since \( \mathcal{C}_{\varepsilon_0}(E^{u}, E^{s}) \) is \( Df \)-forward invariant (and we can assume \( \varepsilon_0< 1 \)), it follows that for any point $x\notin \Lambda(p)$, any $v\in E^{cu}\subset\mathcal{C}_{\varepsilon_0}(E^{u}, E^{s})$, we can write \( v = v_{u} + v_{s} \), where \( v_{u} \in E^{u} \), \( v_{s} \in E^{s} \), and \( v_{u} \neq 0 \). Consequently, by the invariance of $E^s$ and \( Df \)-forward invariance of \( \mathcal{C}_{\varepsilon_0}(E^{u}, E^{s})\) (note that for \( x \notin \Lambda(p) \), \( E^s(x) \) remains invariant),  we have the following inequality: 
{\color{red}\[
\lambda^u\|v_{u}\|\le\|Df(v)\|\le \lambda^u\|v_{u}\|\cdot\sqrt{\varepsilon_0^2+1} \quad \text{and} \quad \frac{1}{\sqrt{\varepsilon_0^2+1}}\le\frac{\|v_{u}\|}{\|v\|}\le1.
\]
(Alternatively, we can directly verify that
$$
\begin{pmatrix}
\lambda^u & 0 \\
* & *
\end{pmatrix}
\begin{pmatrix}
v^u \\
v^s
\end{pmatrix}
=
\begin{pmatrix}
\lambda^u v^u \\
*
\end{pmatrix}
\in \mathcal{C}_{\varepsilon_0}(E^u, E^s).)
$$
From this, it follows that
\[
\frac{\lambda^u}{\sqrt{\varepsilon_0^2+1}}\le \frac{\|Df(v)\|}{\|v\|}.
\]
Since \( \varepsilon_0 \) can be chosen sufficiently small when \( k \) is taken sufficiently large, it follows that by taking $k$ sufficiently large, we can guarantee $\varepsilon_0\le\varepsilon_1$.
It follows that when $x\notin\Lambda(p)$, we have
\[
\frac{\|Df(v)\|}{\|v\|}\ge \frac{1}{\sqrt{\varepsilon_1^2+1}}\cdot\lambda^u.
\]}

When $x\in\Lambda(p)$, for $v=(0,0,1,\varepsilon)\in F^{cu}$, we can check that
\[
\frac{\|Df(v)\|}{\|v\|}=\frac{\sqrt{(\dfrac{\partial P}{\partial c}+\dfrac{\partial P}{\partial d}\cdot\varepsilon)^2+(\lambda^s\varepsilon)^2}}{\sqrt{\varepsilon^2+1}}.
\]
By choosing \( k \) sufficiently large, we can ensure that \( \varepsilon_0 \) is much smaller than \( \varepsilon_1 \), thereby guaranteeing that
\[
\frac{\|Df(v)\|}{\|v\|} = \frac{\sqrt{\left(\dfrac{\partial P}{\partial c} + \dfrac{\partial P}{\partial d} \cdot \varepsilon\right)^2 + (\lambda^s \varepsilon)^2}}{\sqrt{\varepsilon^2 + 1}} \geq \frac{\sqrt{1 - \varepsilon_1}}{\sqrt{\varepsilon_1^2 + 1}}.
\]
(By Lemma~\ref{splitting} and the fact that \( \dfrac{\partial P}{\partial d} \) becomes sufficiently small as \( k \) increases, we can always achieve above.)

Notice that the action of \( f \) on the first torus \( \mathbb{T}^2 \) remains \( D^n \). By the definition of linear Anosov skew-products, it follows that \( f \) is a linear Anosov skew-product. Thus, \( f \) factors over \( D^n \).  For any Gibbs $u$-state $\mu$, by another version of Lemma~\ref{ifonly} with respect to system \(( \mathbb{T}^2 \times \mathbb{T}^2,f) \), $\mu$ is a $c$-Gibbs $u$-state.
Let \(\pi: \mathbb{T}^2 \times \mathbb{T}^2 \to \mathbb{T}^2\) be the projection map defined by \(\pi(x, y) = x\). It follows that for the \(c\)-Gibbs \(u\)-state \(\mu\) on \(\mathbb{T}^2 \times \mathbb{T}^2\), the pushforward measure \(\pi_*(\mu)\) is the Lebesgue measure \(\Leb_{\mathbb{T}^2}\) on the torus \(\mathbb{T}^2\).
\[
\pi_*(\mu) = \Leb_{\mathbb{T}^2} \quad \text{is a property of \( c \)-Gibbs \( u \)-states.}
\]
It follows that
\[
\mu([-2\delta,2\delta]^2\times\T^2)=\pi_*\mu([-2\delta,2\delta]^2)=\Leb_{\mathbb{T}^2}([-2\delta,2\delta]^2)\le\frac{1}{100}\quad \text{and} \quad \Lambda(p)\subset [-2\delta,2\delta]^2\times\T^2.
\]
Then by our setup~\ref{condition} we have
\begin{align*}
\begin{split}
\int\log|\det Df|_{F^{cu}}|d\mu&= \int_{\Lambda(p)}\log|\det Df|_{F^{cu}}|d\mu+ \int_{\T^4\setminus\Lambda(p)}\log|\det Df|_{F^{cu}}|d\mu\\
&\ge[\log\frac{\sqrt{1-\varepsilon_1}}{\sqrt{\varepsilon_1^2+1}}] \cdot \frac{1}{100} + [\log\left(\frac{\lambda^u}{\sqrt{\varepsilon_1^2+1}}\right)] \cdot \frac{99}{100}\\
&>[\log\frac{\sqrt{1-\varepsilon_1}}{\sqrt{\varepsilon_1^2+1}}] \cdot \frac{1}{10} + [\log\left(\frac{\lambda^u}{\sqrt{\varepsilon_1^2+1}}\right)] \cdot \frac{9}{10}>0
\end{split}
\end{align*}

For the second equality, similarly, since $\lambda^u=\frac{1}{\lambda^s}$ {\color{red}and $\mu(f(\Lambda(q)))\le\frac{1}{100}$ }, we can consider the action  under $Df^{-1}$ to establish the desired result. 
\end{proof}

Now, we are ready to proof the last item~\ref{(5)}.

\begin{proof}[Proof of item~\ref{(5)}]
By combining Lemma~\ref{zuihouzhuyao} with the Oseledets multiplicative ergodic theorem (see \cite{Viana}), we can conclude the desired results.
\end{proof}

\begin{cor}\label{proofC}
Recall that \( D = \begin{pmatrix} 2 & 1 \\ 1 & 1 \end{pmatrix} \) and \( A = D^n \times D^m\) on \( \mathbb{T}^2 \times \mathbb{T}^2 \). Then there exists an integer \( n \) such that for the corresponding \( f \) constructed by the above method, there exists a \( C^1 \)-neighborhood \( \mathcal{U}_f \) such that every \( \tilde{f} \in \mathcal{U}_f \) factors over \( D^n \) and has a \( c \)-mixed center.
\end{cor}
\begin{proof}
Let \( n \) and \( m \) satisfy the conditions required for the above construction. Notice that the action of \( f \) on the first torus \( \mathbb{T}^2 \) remains \( D^n \). By the definition of linear Anosov skew-products, it follows that \( f \) is a linear Anosov skew-product. Thus, \( f \) factors over \( D^n \).  Combining item~\ref{(5)} and another version of Lemma~\ref{ifonly} with respect to \( \mathbb{T}^2 \times \mathbb{T}^2 \), we conclude that \( f \) has a \( c \)-mixed center. By choosing \( n \) sufficiently large, the  foliation \( \mathcal{F}^c(f) = \{ \{x\} \times \mathbb{T}^2 : x \in \mathbb{T}^2 \} \) of \( f \) becomes normally hyperbolic. Thus, by the stability theorem in \cite{HPS} and similarly to the proof in \cite[Lemma~5.2]{LiZhang}, there exists a \( C^1 \)-neighborhood \( \mathcal{U}_f \) such that every \( \tilde{f} \in \mathcal{U}_f \) factors over \( D^n \).
Since the property of having a \( c \)-mixed center is open among diffeomorphisms factoring over the same Anosov map as \( f \), it follows that, up to shrinking \( \mathcal{U}_f \), \( \tilde{f} \) also has a \( c \)-mixed center.
\end{proof}

\begin{proof}[Proof of Theorem~C]
Let $f$ as in Corollary~\ref{proofC}. {\color{red}By a slight modification, we replace $I_\varepsilon$ with $\tilde{I_\varepsilon}$ such that  
\begin{itemize}
\item when $(a,b,c,d)\notin U_p\cup U_q$, $\tilde{I_\varepsilon}=\tilde{I_\varepsilon}^{-1}=I$, where $I$ is the idenity;
\item when $(a,b,c,d)\in U_p$, $\tilde{I_\varepsilon}=(a, b, \frac{\tilde{P}(a,b,c,d)}{\lambda^{u}},d)$ on $\Lambda(p)$, where $$\tilde{P}(a,b,c,d)=s(kc)\cdot s(\sqrt{a^2+b^2+d^2})\cdot(c-\lambda^uc-\epsilon c)+\lambda^uc.$$
\item when $(a,b,c,d)\in U_q$, $I_\varepsilon^{-1}=(a, b, c,\lambda^s\cdot \tilde{Q}(a,b,c,d))$ on $\Lambda(q)$,  where 
$$
\tilde{Q}(a,b,c,d)=s(kd)\cdot s(\sqrt{a^2+b^2+c^2})\cdot(d-\frac{1}{\lambda^s}d-\epsilon d)+\frac{1}{\lambda^s}d.
$$
\item $A\circ \tilde{I_\varepsilon}\in\U_f$, where \( \mathcal{U}_f \) is as  in Corollary~\ref{proofC}.
\end{itemize}
}
Let $\tilde{f}=A\circ \tilde{I_\varepsilon}$.
Then, we can confirm that \( p \) is a hyperbolic fixed point with stable index 3, while \( q \) is a hyperbolic fixed point with stable index 1 under \( \tilde{f} \).  Let \( p_g \) and \( q_g \) be the continuation of the saddles of \( p \) and \( q \) for a nearby diffeomorphism \( \tilde{f} \).
Therefore, there exists a \( C^1 \)-neighborhood \( \mathcal{U}_{\tilde{f}} \) of \( \tilde{f} \), contained within \( \mathcal{U}_f \), such that the diffeomorphisms in this neighborhood of \( \tilde{f} \) have hyperbolic fixed points \( p_g \) and \( q_g \), with stable indices 3 and 1, respectively. 

Therefore, \( \tilde{f} \) can serve as the \( f \) described in Theorem C.
\end{proof}

\subsection{Application of Construction}\label{checkuu}

Finally, we point out that the method demonstrated in the construction  can also be applied to modify the linear Anosov skew constructed through Proposition~\ref{pro.con}.

More precisely, 
for example, let \( g: \T^2\times S^1 \to \T^2\times S^1 \) be a $C^2$-partially hyperbolic skew product map, defined by 
$$
 g(x, y) = (B(x), k(x,y)),
$$
 where $ B$ is a hyperbolic linear automorphism with eigenvalues $\sigma_4>1>\sigma_1>0$. 
The eigenvalues $\sigma_4$ and $\sigma_1$  correspond to unstable and stable foliations, $\F_4$ (tangent to $E^{uu}$ at every point )and $\F_1$ (tangent to $E^{ss}$ at every point),  respectively, in the system $(\T^2,B)$. We  assume that the eigenspaces corresponding to distinct eigenvalues are mutually orthogonal.  Using an approach similar to that of Dolgopyat, Viana, and Yang \cite{DVY}, we can guarantee the following conditions:
\begin{itemize}
\item $g$ is $C^2$-closed to map $(B,id)$ (to guarantee $g$ is {\it center bunched}, see \cite{KBAW,CPMS});
\item $k:\T^2\times S^1\to S^1$ is $C^2$.
\item  every Gibbs $u$-state has only negative Lyapunov exponents along $TS^1 $.
\end{itemize}

Next, consider a proper hyperbolic linear automorphism \( C \) on \( \mathbb{T}^2 \) with eigenvalues \( \sigma_3 > 1 > \sigma_2 > 0 \), corresponding to the mutually orthogonal eigenspaces.
The eigenvalues $\sigma_3$ and $\sigma_2$  correspond to unstable and stable foliations, $\F_3$ (tangent to $E^{u}$ at every point) and $\F_2$ (tangent to $E^{s}$ at every point),  respectively, in the system $(\T^2,C)$.
 We can assume that 
\[
 0<\sigma_1 < \sigma_2 <\min\{|Dk_x(y)|:x\in\T^2,y\in S^1\} <1< \max\{|Dk_x(y)|:x\in\T^2,y\in S^1\}< \sigma_3 < \sigma_4.
\]
and
\[ 
 -M(1 - \sigma_3) + \sigma_3 \le \frac{\sigma_4}{2}.
\]
There always exist a hyperbolic periodic point with the stable index $2$ and a hyperbolic periodic point with the stable index $1$ in system $(\T^2\times S^1,g)$. Without loss of generality, assume the hyperbolic periodic point with the stable index $2$ is a fixed point $\tilde{p}$.  We can assume that there exists a constant $\epsilon$ such that $\|Dg|TS^1\|\le 1-\epsilon$ in  a small neighborhood $U_{\tilde{p}}$ of $\tilde{p}$.

Let \(\pi: \mathbb{T}^2 \times S^1 \times \mathbb{T}^2  \to \mathbb{T}^2 \times S^1\) be the projection such that $\pi(x,y,z)=(x,y)$.  By the method we used and using a local  coordinate transformation, we can assume that  there exists a small  open neighborhood $U_{\hat{p}}$ of $\hat{p}$ in system $(\T^2\times S^1\times\T^2, g\times C)$ such that:
\begin{itemize}
\item $\hat{p}=(\tilde{p},0)$ and $C(0)=0$;
\item \([-2\delta, 2\delta]^5 \subset U_{\hat{p}}\), where \(\hat{p}\) is the center of the cube, i.e., the origin $(0,0,0,0,0)$;
\item $\pi(U_{\hat{p}})\subset U_{\tilde{p}}$.
\end{itemize}
We then  modify the action of $g\times C$ in  $U_{\hat{p}}$. In other words,  
define $G$ as follows:
\begin{itemize}
\item when $(a,b,c,d,e)\notin U_{\hat{p}}$, $G=g\times C$;
\item when $(a,b,c,d,e)\in U_{\hat{p}}$, we set $$G(a,b,c,d,e)=(\sigma_4a,\sigma_1b, k(a,b,c), P(a,b,c,d,e), \sigma_2e)$$
 on $\F^{uu}(\hat{p})\times \F^{ss}(\hat{p})\times S^1\times \F^{u}(\hat{p})\times \F^{s}(\hat{p})$, where $$P(a,b,c,d,e)=s(kd)\cdot s(\sqrt{a^2+c^2+b^2+e^2})\cdot(d-\sigma_3d)+\sigma_3d;$$
\end{itemize}
The defined $G$ is a modification of $f\times C$ on $[-2\delta,2\delta]^5$.

Notice that
\begin{equation}\label{checku}
D G(a,b,c,d,e) = \begin{pmatrix}
\sigma_4 & 0 & 0 & 0 & 0 \\[6pt]
0 & \sigma_1 & 0 & 0 & 0 \\[6pt]
\frac{\partial k}{\partial a} & \frac{\partial k}{\partial b} & \frac{\partial k}{\partial c} & 0 & 0 \\[6pt]
\frac{\partial P}{\partial a} & \frac{\partial P}{\partial b} & \frac{\partial P}{\partial c} & \frac{\partial P}{\partial d} & \frac{\partial P}{\partial e} \\[6pt]
0 & 0 & 0 & 0 & \sigma_2
\end{pmatrix}.
\end{equation}
It can be verified that 
\[
\sigma_1<\sigma_2<\|Dg|TS^1\|=\frac{\partial k}{\partial c} \leq 1 - \epsilon < 1 \leq \frac{\partial P}{\partial d} \leq \frac{\sigma_4}{2} < \sigma_4.
\]
Since \( k \) can be chosen sufficiently large and \( k(a,b,\cdot) \) is \( C^2 \)-close to the identity on $S^1$, by selecting appropriate \( k \) and \(  k(a,b,\cdot) \), we can directly verify that  the partial derivatives
$$
\frac{\partial P}{\partial a},\frac{\partial P}{\partial b},\frac{\partial P}{\partial c},\frac{\partial P}{\partial e}, \frac{\partial k}{\partial a} \quad \text{and} \quad \frac{\partial k}{\partial b}
$$
can simultaneously be made arbitrarily close to $0$.
It follows from the forward invariance of the cone that a partially hyperbolic splitting exists:
\[
T(\T^2\times S^1\times\T^2)=F^{uu}_G \oplus_\succ F^{cu}_G \oplus_\succ (F^{cs}_G \oplus_\succ F^{ss}_G),
\]
 such that 
\begin{itemize}
\item $F^{uu}_G\subset E^{uu}\oplus TS^1\oplus E^u$, $F^{cu}_G=E^u$, $F^{cs}\subset TS^1\oplus E^u$, $F^{ss}_G\subset E^{ss}\oplus E^s\oplus TS^1\oplus E^u$ and $\dim(F^{ss}_G) = 2$;
\item \( F^{cu}_G \) is not uniformly expanding (similar to item~\ref{(4)}) such that \( \| Dg |_{F^{cu}} \| \ge 1 \) when $x\in U_{\hat{p}}$ and \( \| Dg |_{F^{cu}} \| > 1 \) when $x\notin U_{\hat{p}}$.
\item $F^{cs}_G$ (can be as $T\tilde{S^1}$ in Lemma~\ref{equvilence}) is not uniformly contracting (by our assumption on system $(\T^2\times S^1,g)$) such that
\[
  T\tilde{S^1}\subset TS^1\oplus F^{cu} \quad \text{and} \quad Df(TS^1\oplus  F^{cu})=TS^1\oplus  F^{cu}.
\]
\end{itemize}

 Since \( k(a,b,c) \) is independent of the parameters \( d \) and \( e \) on \( [-2\delta, 2\delta]^5 \), and the map projected onto \( \T^2 \times S^1 \) remains \( g \), we have \( \pi(G) = g \).
By invoking Lemma~\ref{equvilence}, it follows that every Gibbs $u$-state has only negative Lyapunov exponents along $T\tilde{S^1}$ in system $(\T^2\times S^1\times\T^2, G)$.
By combining this with $\int \log |\det Dg| d\mu>0$ for any Gibbs $u$-state in system $(\T^2\times S^1\times\T^2, G)$, we can conclude that the corresponding \( F^{cu}_G \) is not uniformly expanding, and \( G \) exhibits the mixed property.

\begin{rk}
We point out that by appropriately constructing \( g \), we can provide an example of variations in the physical measures of the system \( (\mathbb{T}^2 \times S^1 \times \mathbb{T}^2, G) \), where there exists a partially hyperbolic splitting 
\[
T(\mathbb{T}^2 \times S^1 \times \mathbb{T}^2) = F^{uu}_G \oplus_\succ F^{cu}_G \oplus_\succ (F^{cs}_G \oplus_\succ F^{ss}_G),
\]
and the corresponding \( E^{cu} = F^{cu}_G \) and \( E^{cs} = F^{cs}_G \oplus F^{ss}_G \) are both non-uniform.

We can consider a skew product system \( g(x, y) = (B(x), k_x(y)) \) with eight invariant tori \( \mathbb{T}^2 \times A_i \), where \( i = 1, 2, 3, 4, 5, 6, 7, 8 \), and eight physical measures are supported on these eight invariant tori. We modify the action of the map \( g \times C \) on a neighborhood $U$ of a fixed points on the invariant torus \( \mathbb{T}^2 \times A_2 \times \mathbb{T}^2 \), where $U\cap(\cup_{i=3,4,5,6,7,8,1}\mathbb{T}^2 \times A_i \times \mathbb{T}^2) =\emptyset$. Then, for the map acting from the invariant torus \( \mathbb{T}^2 \times A_3 \times \mathbb{T}^2 \) to \( \mathbb{T}^2 \times A_1 \times \mathbb{T}^2 \) (with the \( S^1 \) direction taken counterclockwise), the action remains the same as in the original \( g \times C \), provided that the part from \( \mathbb{T}^2 \times A_3 \times \mathbb{T}^2 \) to \( \mathbb{T}^2 \times A_1 \times \mathbb{T}^2 \) does not include \( \mathbb{T}^2 \times A_2 \times \mathbb{T}^2 \). By perturbing the neighborhood of the unstable manifold of a fixed point, where the fixed point is contained in the skeleton intersecting \( \cup_{i=4,5,6,7,8}\mathbb{T}^2 \times A_i \times \mathbb{T}^2 \), the cardinality of the skeleton decreases by one after the perturbation. This guarantees that the number of physical measures will also decrease by one.
\end{rk}

\section{Conflict of Interest}

The author declares that there is no conflict of interest regarding the publication of this paper.

\section{Data Availability}

No data was generated or analyzed during the study.

\smallskip

E-mail address: zhanghangyue@nju.edu.cn

\end{document}